\documentclass{article}

\usepackage{amsmath, amssymb, latexsym}

\usepackage{amsthm}

\usepackage{graphs}

\newtheorem{theorem}{Theorem}[section]
\newtheorem{proposition}[theorem]{Proposition}
\newtheorem{lemma}[theorem]{Lemma}
\newtheorem{corollary}[theorem]{Corollary}

\theoremstyle{definition}
\newtheorem{definition}[theorem]{Definition}
\theoremstyle{remark}
\newtheorem{remark}[theorem]{Remark}
\theoremstyle{remark}
\newtheorem{example}[theorem]{Example}
\theoremstyle{remark}

\newcommand{\e}{\varepsilon}
\newcommand{\R}{\mathbb R}
\newcommand{\ov}{\overline}
\newcommand{\wh}{\widehat}
\newcommand{\N}{\mathbb N}

\newcommand{\E}{\mathcal E}
\def\fr{{\rm freq}}
\numberwithin{equation}{section}
\def\be{\begin{equation}}
\def\ee{\end{equation}}
\def\Nk{{\mathcal N}}

\begin{document}

\title{Finite rank Bratteli diagrams: structure of invariant measures}

\date{}

\author{{\bf S.~Bezuglyi}\\
Institute for Low Temperature Physics, Kharkov, Ukraine \\
bezuglyi@ilt.kharkov.ua
\\
{\bf J.~Kwiatkowski}\footnote{The research of J.K was supported
by grant MNiSzW N N201384834.}\\
University of Warmia and Mazury, Olsztyn, Poland\\
jkwiat@mat.uni.torun.pl
\\
 {\bf K. Medynets}\\
Ohio State University, Columbus, USA\\
medynets@math.ohio-state.edu
\\
{\bf B. Solomyak}\footnote{ B.S. was supported in part by NSF grants
 DMS-0654408 and DMS-0968879.}\\
University of Washington, Seattle, USA\\
solomyak@math.washington.edu}

\maketitle

%
\begin{abstract} We consider Bratteli diagrams of finite rank (not necessarily simple) and ergodic invariant measures
with respect to the cofinal equivalence relation on their path spaces. It is shown that every ergodic invariant measure (finite or
``regular'' infinite) is obtained by an extension from a simple subdiagram. We further investigate quantitative properties of these measures,
which are mainly determined by the asymptotic behavior of products of incidence matrices.
A number of sufficient conditions for unique ergodicity are obtained.
One of these is a condition of exact finite rank, which parallels a similar notion in measurable dynamics.
Several examples illustrate the broad range of possible behavior of finite type diagrams and invariant measures on them.
We then prove that the Vershik map on the path space of an exact finite rank diagram cannot be strongly mixing, independent of the ordering.
On the other hand, for the so-called ``consecutive'' ordering, the Vershik map is not strongly mixing on all finite rank diagrams.
\end{abstract}

\noindent{\small MSC: 37B05, 37A25, 37A20.}

%
%
\section{Introduction} Bratteli diagrams, which originally appeared in the theory of operator algebras, turned out to be a powerful method
for the study of dynamical systems in ergodic theory and Cantor dynamics. Every minimal and even aperiodic homeomorphism of a Cantor set
can be represented as a Vershik map acting on the path space of a
Bratteli diagram (\cite{herman_putnam_skau:1992},
\cite{medynets:2006}). The main object of our study is the class of finite rank Bratteli diagrams, i.e., the diagrams whose vertex set at each level
  is uniformly bounded or, equivalently (after an easy reduction),
  with the same number of vertices at each level.
It is worth pointing out that, in contrast to most papers on Cantor dynamics and
Bratteli diagrams, our interest is focused on general, not necessarily simple, Bratteli diagrams.
In this context, the present paper is a natural continuation of our previous work
  \cite{bezuglyi_kwiatkowski_medynets_solomyak:2010}
devoted to the study of invariant measures and the structure of stationary non-simple Bratteli diagrams.


Our main goal is to describe the structure of invariant (with respect to Vershik maps or, more generally, the cofinal equivalence relation)
Borel non-atomic measures on finite rank Bratteli diagrams.
One of our motivations is the application to the
classification theory of Cantor dynamical systems up to orbit equivalence \cite{giordano_putnam_skau:1995}.
Namely, the knowledge of supports of invariant measures, the number of minimal components and ergodic measures, and the
measure values on clopen sets are useful for distinguishing
 non-isomorphic or non-orbit equivalent homeomorphisms. Observe also that invariant measures are in one-to-one
correspondence with states of dimension groups determined by the diagram. In 
\cite{goodearl_handelman:82}, Goodearl and Handelman studied the problem
of state extension for extensions of dimension groups.
Some structural results on finite rank dimension groups are presented in 
\cite{effors_shen:1979} and \cite{effros_shen:1981}. 

The choice of finite rank systems is based on their ``relative''
combinatorial simplicity and, at the same time, on the intriguing properties and non-trivial dynamical behavior, see, for example,
\cite{cortez_durand_host_maass:2003}, \cite{bressaud_durand_maass:2009}, \cite[Chapter 6]{berthe:2010},
and \cite{bezuglyi_kwiatkowski_medynets_solomyak:2010}.
We observe that substitution dynamical systems, minimal interval exchange transformations, and generalized Morse sequences
(Example \ref{ExampleGeneralizedMorseSequence})
belong to this class of systems \cite{durand_host_scau:1999},
 \cite{bezuglyi_kwiatkowski_medynets:2009},
 \cite{gjerde_johansen:2002}. Conversely, every Vershik map on a
 finite rank diagram is either an odometer or a subshift on a finite alphabet (\cite{downarowich_maass:2008}, \cite{bezuglyi_kwiatkowski_medynets:2009}).
However, to the best of our knowledge, it is still unknown what kind of subshifts can arise on this way.

Every Bratteli diagram is completely determined by a
sequence of incidence matrices. The {\em Vershik map} is defined once we equip the diagram with an appropriate order.
However, if one is interested in the properties of the corresponding dynamical system that do not depend on the order of points in its orbits
(like invariant measures, minimal components, etc.), then it suffices to study the incidence matrices only.
We show that the structure of the set of invariant measures can be derived from the ``growth rate'' of entries
of incidence matrix products. In our previous work \cite{bezuglyi_kwiatkowski_medynets_solomyak:2010} we
applied a similar idea --- studying the asymptotic growth of powers of a single matrix
(geometric Perron-Frobenius theory) --- to describe ergodic invariant measures for stationary diagrams. In the non-stationary case, we consider the
non-homogeneous products of (see
 \cite{seneta:book:1981} and \cite{hartfiel:book:2002} for the essence of the theory) to study the dynamical properties.
 For example, a simple Bratteli diagram is uniquely ergodic if
and only if the rows in backward products of incidence matrices become nearly proportional (see \cite{fisherUniqueErgodicity:preprint} or Theorem
\ref{theoremUniqueErgodicityInTermsTau} below). The property of near proportionality can be checked by
methods of linear algebra (the technique of Birkhoff contraction coefficient),
which gives a purely algebraic criterion of unique ergodicity for Vershik maps.

Our main results and the paper organization are as follows.

In Section \ref{SectionPreliminary} we give the definition and
necessary notation of Bratteli diagrams and Vershik maps.
We explain the relation between invariant measures and products of incidence matrixes. We also show
that every finite rank diagram can be transformed into
a ``canonical'' block-triangular form which is
convenient for describing the structure of invariant measures.

In Section \ref{SectionStructureInvariantMeasures} we establish general structural
properties of invariant measures on finite rank Bratteli diagrams. We prove that any finite rank Bratteli diagram admits only a finite
number of ergodic (both finite and ``regular'' infinite) measures and every
ergodic measure is, in fact, an extension of a finite ergodic
measure from a simple subdiagram (Theorem
\ref{TheoremGeneralStructureOfMeasures}). This subdiagram has the property that the measures of towers specified by the vertices from the subdiagram
 are bounded away from zero. We note that this condition on a subdiagram corresponds to the definition
 of exact finite rank in measurable dynamics \cite{ferenczi:1997}. As a corollary,
 we prove that all diagrams of exact finite rank (Definition \ref{DefinitionExact}) are uniquely
ergodic. This fact can be considered as a version of Boshernitzan's
theorem \cite{boshernitzan:1992} proved in the context of symbolic dynamics.
It is interesting to note that Boshernitzan's condition for symbolic systems has been recently used to prove
uniform convergence in the multiplicative ergodic theorem,
which has applications to the spectral properties of Schr\"odinger operators \cite{damanik_lenz:2006}.

Section \ref{SectionUniqueErgodicity} collects results, which are mostly known but scattered in the literature, on unique ergodicity for
Vershik maps on simple (finite rank)
Bratteli diagrams. In particular, we prove a criterion for unique ergodicity, which first appeared (in a slightly different form and with a different proof)
in the work of A. Fisher \cite{fisherUniqueErgodicity:preprint}.
We also list several easily computable sufficient conditions of unique ergodicity. As an example, we show how these
conditions can be reformulated in symbolic terms when applied to generalized Morse sequences.
We note that algebraic conditions of (non)-unique ergodicity were also considered in the paper \cite{ferenczi_fisher_talet}
for diagrams with two and three vertices at each level.
All necessary results concerning matrix products and, especially,
the notion of Birkhoff contraction coefficient are also presented in this section.

In Section \ref{SectionQuantitativeAnalysis} we study the asymptotic growth rate of tower heights and measures of tower bases. We show that
for exact finite rank diagrams the measures of tower bases are (asymptotically) reciprocal to the tower heights.
In the case when the tower heights have the same asymptotic behavior,
 this growth can be estimated by the norm of the product of incidence matrices. These results can be viewed as ``adic'' counterparts of some results in
quantitative recurrence theory; see, for example, \cite{boshernitzan:1993} and \cite{galatolo_kim:2007}. We present an example of a diagram
showing that the exact finite rank does not guarantee the same asymptotic growth for tower heights.
On the other hand, if a diagram determined by matrices
$\{F_n\}$ satisfies the ``compactness'' condition, $m_n/M_n\geq c > 0$, where $m_n$ and $M_n$ are  the smallest and the largest entries of $F_n$
respectively, then the diagram has exact finite rank and the tower heights  have the same asymptotic growth.

In Section \ref{SectionMeasureExtension} we focus on non-simple diagrams and
further study the construction of extension of invariant measures from a simple subdiagram developed in Section \ref{SectionStructureInvariantMeasures}.
Our main question here is how to determine (in algebraic terms) when such an extension remains a finite measure.
We provide several sufficient conditions for that and give illustrative examples.
In the last part of the section we consider such an extension for the diagrams that have only a finite number of distinct incidence matrices
(we call such diagrams ``linearly recurrent''). For such diagrams the question of finiteness of the extension can be
reduced to the comparison of two numbers. 

In Section \ref{SectionAbsenceStrongMixing}, we apply the properties of invariant measures to prove that any Vershik map on a diagram of exact finite rank
(for any order) is not strongly mixing. This result generalizes the corresponding facts on linearly recurrent systems \cite{cortez_durand_host_maass:2003}
and substitution systems \cite{dekking_keane:1978}, \cite{bezuglyi_kwiatkowski_medynets_solomyak:2010}. We then show that the exactness requirement can be
dropped if we have a ``consecutive
 ordering'' on the diagram. Note that Bratteli diagrams corresponding to minimal interval exchange transformations have consecutive orderings
 \cite{gjerde_johansen:2002}. The absence of mixing for interval exchanges was proved by A. Katok \cite{katok:1980}, and our methods have some common
features with those of \cite{katok:1980}.

%
%
\section{Bratteli diagrams: basic constructions}\label{SectionPreliminary}

In this section we collect the notation and basic definitions that
are used throughout the paper. Since the notion of Bratteli
diagrams and the related notion of Vershik transformation have been
discussed in numerous recent papers, they might be considered as
almost classical nowadays, so we avoid giving
detailed definitions. An interested reader may consult the papers
\cite{herman_putnam_skau:1992}, \cite{giordano_putnam_skau:1995},
\cite{durand_host_scau:1999}, \cite{medynets:2006},
\cite{bezuglyi_kwiatkowski_medynets:2009},
\cite{bezuglyi_kwiatkowski_medynets_solomyak:2010}, and references
therein for all details concerning Bratteli diagrams and Vershik
maps. We only give here some basic definitions in order to fix our
notation.

\subsection{Bratteli diagrams}
\begin{definition}\label{Definition_Bratteli_Diagram} A {\it Bratteli diagram} is
an infinite graph $B=(V,E)$ such that the vertex set
$V=\bigcup_{i\geq 0}V_i$ and the edge set $E=\bigcup_{i\geq 1}E_i$
are partitioned into disjoint subsets $V_i$ and $E_i$ such that

(i) $V_0=\{v_0\}$ is a single point;

(ii) $V_i$ and $E_i$ are finite sets;

(iii) there exist a range map $r$ and a source map $s$ from $E$ to
$V$ such that $r(E_i)= V_i$, $s(E_i)= V_{i-1}$, and
$s^{-1}(v)\neq\emptyset$, $r^{-1}(v')\neq\emptyset$ for all $v\in V$
and $v'\in V\setminus V_0$.
\end{definition}

The pair $(V_i,E_i)$ or just $V_i$ is called the $i$-th level of the diagram $B$.
 A finite or infinite sequence of edges $(e_i : e_i\in E_i)$ such
that $r(e_{i})=s(e_{i+1})$ is called a {\it finite} or {\it infinite
path}, respectively. We write $e(v,v')$ to denote a path $e$ such
that $s(e)=v$ and $r(e)=v'$. For a Bratteli diagram $B$, we denote
by $X_B$ the set of infinite paths starting at the vertex $v_0$. We
endow $X_B$ with the topology generated by cylinder sets
$U(e_1,\ldots,e_n)=\{x\in X_B : x_i=e_i,\;i=1,\ldots,n\}$, where
$(e_1,\ldots,e_n)$ is a finite path from $B$. Then $X_B$ is a
0-dimensional compact metric space with respect to this topology.

Given a Bratteli diagram $B=(V,E)$, the incidence matrix
$F_{n}=(f^{(n)}_{v,w}),\ n\geq 1,$ is a $|V_{n+1}|\times |V_n|$
matrix whose entries $f^{(n)}_{v,w}$ are equal to the number of
edges between the vertices $v\in V_{n+1}$ and $w\in V_{n}$, i.e.,
$$
 f^{(n)}_{v,w} = |\{e\in E_{n+1} : r(e) = v, s(e) = w\}|.
$$
(Here and thereafter $|A|$ denotes the cardinality of the set $A$.)
We notice that $F_0$ is a vector. We assume usually that $F_0 =
 (1,...,1)^T$.

Observe that every vertex $v\in V $ is connected to $v_0$ by a
finite path and the set $E(v_0,v)$ of all such paths is finite. Set
$h_v^{(n)}=|E(v_0,v)|$ where $v\in V_{n}$. Then
\begin{equation}\label{towerHeight}
h_v^{(n+1)}=\sum_{w\in V_{n}}f_{v,w}^{(n)}h^{(n)}_w
\end{equation}
or
\begin{equation}\label{heights}
h^{(n+1)}=F_{n}h^{(n)}
\end{equation}
where $h^{(n)}=(h_w^{(n)})_{w\in V_n}$.

Together with the sequence of incidence matrices $\{F_n\}$ we will
use the sequence of matrices $\{Q_n\}$ where the entries
$q_{v,w}^{(n)}$ of $Q_n$ are defined by the formula:
\begin{equation}\label{StochasticMatrix}
q_{v,w}^{(n)} = f_{v,w}^{(n)}\frac{h^{(n)}_w}{h_v^{(n+1)}},\ \ n\geq
1.
\end{equation}
It follows from (\ref{towerHeight}) that every $Q_n$ is a \textit{stochastic}
matrix.

It is not hard to show that for a given sequence of non-negative rational stochastic $d\times d$ matrices
$\{Q_n \}$ there exists a Bratteli diagram $B$ with incidence matrices $\{F_n\}$ whose entries satisfy
(\ref{StochasticMatrix}). The sequence $\{F_n\}$ is not uniquely determined:
 matrices $F_n$ and $pF_n,\ p\in \N,$ correspond to the same stochastic matrix $Q_n$.

For $w\in V_n$, the set $E(v_0,w)$ defines the clopen subset
$$
X^{(n)}_w=\{x=(x_i)\in X_B : r(x_n)=w \}.
$$
The sets $\{X^{(n)}_w : w\in V_n\}$ form a clopen partition of
$X_B,\ n\geq 1$. Analogously, each finite path $\overline
e=(e_1,\ldots,e_n)\in E(v_0,w)$ determines the clopen subset
$$
X^{(n)}_w(\overline e)=\{x=(x_i)\in X_B : x_i=e_i,\;i=1,\ldots,n \}.
$$
These sets form a clopen partition of $X_w^{(n)}$. We will use also
the notation $[\overline{e}]$ for the clopen set
$X^{(n)}_w(\overline e)$ when it does not lead to a confusion. The
base of the tower $X^{(n)}_w$ is denoted by $B_n(w)$. (In fact, this
means that an order is specified on $E(v_0,w)$. But since, in most cases,
order is inessential for us, the subset $B_n(w)$ may
be represented by any finite path from $E(v_0,w)$).

\begin{definition}
A Bratteli diagram $B=(V,E)$ is called {\it simple} for any level $n$
there is $m>n$ such that each pair of vertices
 $(v,w)\in (V_n,V_m)$ is connected by a finite path.
\end{definition}

\begin{definition}
For a Bratteli diagram $B$, the {\it tail (cofinal) equivalence}
 relation $\mathcal E$ on the path space $X_B$ is defined as
$x\mathcal Ey$ if $x_n=y_n$ for all $n$ sufficiently large.
\end{definition}

\begin{remark}
Given a dynamical system $(X,T)$, a Bratteli diagram
is constructed by a sequence
of Kakutani-Rokhlin partitions generated by $(X,T)$ (see
\cite{herman_putnam_skau:1992} and \cite{medynets:2006}). The
$n$-th level of the diagram corresponds to the $n$-th
Kakutani-Rokhlin partition
 and the number $h_w^{(n)}$ is the height of the $T$-tower labeled by
the symbol $w$ from that partition.
\end{remark}

Throughout the paper we will use the telescoping
procedure for a Bratteli diagram. Roughly speaking, in order to telescope a
Bratteli diagram, one takes a subsequence of levels
 $\{n_k\}$ and considers the set of all finite paths between the consecutive levels $\{n_k\}$ and $\{n_{k+1}\}$ as new edges.
 A rigorous definition of telescoping can be found in many papers on Bratteli diagrams, for example, in \cite{giordano_putnam_skau:1995}.

Telescoping, together with the obvious level-preserving graph isomorphism, generate an equivalence relation on the Bratteli diagrams.
Two diagrams in the same class are called {\em isomorphic}.

%
%
%

\subsection{Finite rank Bratteli diagrams}

\begin{definition} A Bratteli diagram that has a uniformly bounded number of vertices at each level is called a diagram of
{\it finite rank}.
\end{definition}

The next theorem shows that each finite rank Bratteli diagram can be
isomorphically transformed into a canonical block-triangular form, which
gives a natural decomposition of $X_B$ into a finite number of
tail-invariant subsets.

\begin{theorem}\label{structure finite rank diagrams}
Any Bratteli diagram of
finite rank is isomorphic to a diagram whose incidence matrices
$\{F_n\}_{n\geq 1}$ are as follows:

\begin{equation}\label{Frobenius Form: General}
F_n =\left(
  \begin{array}{ccccccc}
    F_1^{(n)} & 0 & \cdots & 0 & 0 & \cdots & 0 \\
    0 & F_2^{(n)} & \cdots & 0 & 0 & \cdots & 0 \\
    \vdots & \vdots & \ddots & \vdots & \vdots & \cdots& \vdots \\
    0 & 0 & \cdots & F_s^{(n)} & 0 & \cdots & 0 \\
    X_{s+1,1}^{(n)} & X_{s+1,2}^{(n)} & \cdots & X_{s+1,s}^{(n)} & F_{s+1}^{(n)} & \cdots & 0 \\
    \vdots & \vdots & \cdots & \vdots & \vdots & \ddots & \vdots \\
    X_{m,1}^{(n)} & X_{m,2}^{(n)} & \cdots & X_{m,s}^{(n)} & X_{m,s+1}^{(n)} & \cdots & F_m^{(n)} \\
  \end{array}
\right).
\end{equation}
For every $n\geq 1$, the matrices $F_i^{(n)},\ i=1,...,s $, have strictly
positive entries and the matrices $F_i^{(n)},\ i=s+1,...,m$, have either
all strictly positive or all zero entries. For every fixed $j = s+1,...,m$,
there is at least one non-zero matrix $X_{j,k}^{(n)},\ k = 1,...,j-1$.
\end{theorem}

\begin{proof} Let $B$ be a finite rank Bratteli diagram. By telescoping,
we obtain that $|V_n|= d$ for all $n\geq 1$. It follows from Proposition 4.6 of \cite{bezuglyi_kwiatkowski_medynets:2009}
that $B$ has finitely many minimal components with respect to the tail equivalence relation, say, they are $Z_1,...,Z_s$. Denote
$$
W_n(i) = \{r(x_n)\in V_n : x = (x_n) \in Z_i\}, \ \ i = 1,...,s.
$$

\noindent
{\bf Claim}: For any $i_1 \neq i_2$, there exists $N$ such that for all $n \geq N$
$$
W_n(i_1) \cap W_n(i_2) = \emptyset,\ \ i_1, i_2 = 1,...,s.
$$

To prove the claim, we fix $Z_i$ and consider the subdiagram $B_i$ of $B$ which is formed by
the vertex set $W(i) = \bigcup_{n\geq 1}W_n(i)$ and the edges induced by all paths from $Z_i$. Then $B_i$ is a simple Bratteli diagram.

Suppose now that the contrary holds, i.e., there exist distinct $i_1$ and $i_2$
and a sequence $\{n_k\}$ such that $W_{n_k}(i_1) \cap W_{n_k}(i_2) \neq \emptyset$.
Let $\{v_{n_k}\}$ be a sequence of vertices which is chosen from $W_{n_k}(i_1) \cap W_{n_k}(i_2)$.
Without loss of generality, we may assume that $n_{k+1} - n_k > 2$.
By simplicity of subdiagrams $B_{i_1}$ and $B_{i_2}$, there are finite paths $\ov e_k(1)$
and $\ov e_k(2)$ connecting the vertices $v_{n_k}$ and $v_{n_{k+1}}$ and such that $\ov e_k(1)$
and $\ov e_k(2)$ belong to $B_{i_1}$ and $B_{i_2}$, respectively. Therefore, there exist
infinite paths $x\in Z_{i_1}$
(obtained as a concatenation of $\ov e_1(k)$) and $y\in Z_{i_2}$ (obtained as a concatenation of
$\ov e_2(k)$) which go through the vertices $v_{n_k}$ for every $k\geq
1$. Thus, for every $k\geq 1$, there exists a path $x_k\in
Z_{i_1}$ cofinal to $x$ which coincides with the first $n_k$
edges of $y$. This implies that $\mbox{dist}(x_k,y)\to 0$ as
$k\to\infty$. Hence $\mbox{dist}(Z_{i_1}, Z_{i_2})$ = 0, which
is impossible. To complete the proof of the claim, we use a standard argument based on finiteness of the set of minimal components.
\medskip

By telescoping the diagram $B$, we may assume that $W_n(i_1) \cap W_n(i_2) = \emptyset\ (i_1 \neq i_2)$
for all $n\geq 1$. One can also regroup the vertices at each level so that the sets $W_n(1),...,W_n(s)$ are enumerated from left to right.

Choose a positive constant $\delta$ so that $\mbox{dist}(Z_{i}, Z_j)
\geq \delta, i\neq j$. Again using the method of telescoping, we can
easily reduce the general case to that when no edges between
vertices from different minimal components exist. Hence we have
constructed the collection of simple subdiagrams $B_i$
 with incidence matrices $\{F_i^{(n)}\},\ i=1,\ldots,s$. Further
 telescoping the diagram we may ensure that each matrix $F_i^{(n)}$
 has strictly positive entries.

Next, we consider the subdiagram $B'$ of $B$ whose vertex set $V'$ is $\bigcup_n V'_n$ where
$V'_n = V_n \setminus \bigcup_{i=1}^s W_n(i)$ and the edge set $E'$ consists of the edges
that connect vertices from $V'$ only. In other words, we temporarily ignore the set of
edges that link vertices from $B'$ and those from $B_i,\ i=1,\ldots,s$.
Then $B'$ is a finite rank Bratteli diagram whose rank is strictly less than the rank of $B$. We can
 apply the described above procedure to find all minimal components of $B'$. In a finite number of
such steps, we obtain all simple subdiagrams of $B$ that correspond to non-zero matrices from the
set $\{F_j^{(n)}\},\ j = s+1,\ldots,m.$
It may happen that there will be some vertices at infinitely many levels that do not belong to the constructed
simple subdiagrams. This means that after appropriate telescoping the corresponding incidence matrices $F_j^{(n)}$
 must be either zero or strictly positive.

To finish the proof, we return to $B$ and restore all edges that
have been temporarily removed. They will now connect some vertices
from different subdiagrams $B'_j, j = s+1,\ldots,m$ and also connect
them with some vertices from $B_i, i = 1,\ldots,s$. This set of edges
determines the matrices $X_{i,j}^{(n)}$. Certainly, some of these
matrices may be zero. But if one fixes a row $i \in \{s+1,\ldots,m\}$, then
at least one matrix from the collection $\{X_{i,j}^{(n)}\}$ is non-zero.
\end{proof}


\subsection{Invariant measures}

\begin{definition} Let $B$ be a Bratteli diagram. By a {\it finite measure} on $B$ we always
mean a Borel non-atomic (not necessarily probability) measure on
$X_B$. For an {\it infinite $\sigma$-finite measure} $\mu$ on $X_B$, we assume that $\mu$ takes finite (non-zero)
 values on some clopen sets.
\end{definition}

\begin{definition} Given a Bratteli diagram $B=(V,E)$, a measure $\mu$ on $X_B$ is called {\it invariant}
if $\mu([\overline e])=\mu([\overline e'])$ for
 any two finite paths $\overline e$ and
 $\overline e'$ with the same range. In other words, $\mu(X^{(n)}_w(\overline e)) = \mu(X^{(n)}_w(\overline e'))$
 for any $n\geq 1$ and $w\in V_n$.
\end{definition}

\begin{remark} The measure $\mu$ is
invariant on $B$ if and only if it is invariant with respect to the
cofinal equivalence relation $\mathcal E$.
\end{remark}

\begin{definition} An invariant measure $\mu$ is {\it ergodic} for the diagram
$B$ (or {\it $B$-ergodic}) if it is ergodic with respect to the
cofinal equivalence relation $\mathcal E$.

If a Bratteli diagram $B$ admits a unique invariant probability measure,
then $B$ is called {\it uniquely ergodic}.
\end{definition}

The next theorem which was proved in
\cite{bezuglyi_kwiatkowski_medynets_solomyak:2010} shows that the
simplex of invariant measures is completely determined by the sequence
of incidence matrices of the diagram. To state the theorem, we will
need to introduce the following notation.

For $x=(x_1,\ldots,x_N)^T\in \R^N$, we will write $x\ge 0$ if
$x_i\ge 0$ for all $i$, and consider the positive cone
$\R^N_+=\{x\in\R^N : x\geq 0\}$. Let
$$
C_k^{(n)}:= F_k^T\cdots F_n^T\left(\R_+^{|V_{n+1}|}\right),\ \ 1 \le k \le n.
$$
Clearly, $\R_+^{|V_k|}\supset C_k^{(n)} \supset C_k^{(n+1)}$ for all $n\ge 1$. Let
$$
C_k^{\infty} = \bigcap_{n\ge k} C_k^{(n)}, \ k\ge 1.
$$
Observe that $C_k^\infty$ is a closed non-empty convex subcone of $\R_+^{|V_k|}$.
It also follows from these definitions that
\begin{equation}\label{cones}
    F_k^T C_{k+1}^\infty = C_k^{\infty}.
\end{equation}

\begin{theorem}\cite[Theorem 2.9]{bezuglyi_kwiatkowski_medynets_solomyak:2010}
\label{Theorem_measures_general_case} Let $B = (V,E)$ be a Bratteli diagram such
that the tail equivalence relation $\mathcal E$ on $X_B$ is aperiodic. If $\mu$ is an
invariant measure with respect to the tail equivalence relation $\E$,
then the vectors $p^{(n)}=(\mu(X_w^{(n)}(\ov{e})))_{w\in
V_n}$, $\ov{e}\in E(v_0,w)$, satisfy the following conditions for $n\geq 1$:

(i) $p^{(n)} \in C_n^\infty$,

(ii) $F_n^T p^{(n+1)} = p^{(n)}$.

Conversely, if a sequence of vectors $\{ p^{(n)}\}$ from
$\mathbb R^{|V_n|}_+$ satisfies condition (ii), then there
exists a non-atomic finite Borel $\mathcal E$-invariant measure $\mu$ on
$X_B$ with $p_w^{(n)}=\mu(X_w^{(n)}(\overline e))$ for all $n\geq 1$
and $w\in V_n$.

The $\E$-invariant measure $\mu$ is a probability measure if and only if

(iii) $\sum_{w\in V_n}h_w^{(n)}p_w^{(n)}=1$ for $n=1$,

\noindent in which case this equality holds for all $n\ge 1$.
\end{theorem}

\begin{remark}\label{RemarkSimplexFirstLevelUnique} It was also proved in
\cite[Theorem 3.8]{bezuglyi_kwiatkowski_medynets_solomyak:2010} that
for stationary Bratteli diagrams the sequence of vectors
$\{p^{(n)}\}$ which determines an invariant measure can be
completely restored by the initial distribution vector $p^{(1)}$.
One can construct an example when this result fails for general diagrams.
However, for diagrams of finite rank we can still telescope the
diagram in such a way that any two different invariant measures
$\mu$ and $\nu$ can already be distinguished on the first level,
i.e. the corresponding vectors $p^{(1)}$ are distinct.

Indeed, it follows from the proof of Proposition
\ref{PropositionSolomyakNote} (see below) that the number of extreme rays of
$C_k^\infty$ stabilizes to the number of ergodic measures as
$k\to\infty$. By telescoping we may assume that this already holds
for every $k$. By (\ref{cones}), we see that the linear map
$F_k^T$ sends extreme rays onto extreme rays. Thus, $F_k^T$ is a
bijection of the cones $C_{k+1}^\infty$ and $C_k^{\infty}$ for all
$k$ proving the claim.
\end{remark}

In the next result we apply Theorem
\ref{Theorem_measures_general_case} to a finite rank Bratteli
diagram to show that any such a diagram has a finite number of
ergodic measures. This result may be considered ``folklore'': it was mentioned
in \cite{bressaud_durand_maass:2009} for simple diagrams without a proof, and was certainly known to A. Vershik much earlier
[D. Handelman, personal communication].

\begin{proposition}\label{PropositionSolomyakNote} Let $B$ be a Bratteli diagram
of finite rank. Suppose that the number of vertices at each level is
bounded by $d$.
 Then $B$ has no more than $d$ invariant ergodic probability measures.
\end{proposition}

\begin{proof} We will use Theorem 2.1 from \cite{pullman:1971}:
\medskip

{\it Let $\{C_n\}$ be a sequence of finitely generated cones such that $C_n \supset C_{n+1}$ for all $n \geq 1$.
If for all sufficiently large $n$ the cone $C_n$ is finitely generated by at most $d$ rays, then $C = \bigcap_n C_n$
is also a finitely generated cone by at most $d$ rays
 (the number of generating rays is called the size of the cone)}.
\medskip

We can apply Pullman's theorem to the sequence of cones $\{C_k^{(n)}\}_n$ for all $k \geq 1$
and conclude that the cones $C_k^\infty$ are finitely generated of size not greater than $d$.
It follows from (\ref{cones}) that size($C_k^\infty) \leq $ size($C_{k+1}^\infty$).
Hence the sizes must stabilize: size($C_k^\infty) = m$ for all $k \geq N_0$. Then $B$ has $m$
ergodic invariant probability measures. In fact, it easily follows from Theorem \ref{Theorem_measures_general_case}
that there is a 1-1 correspondence between $\E$-invariant measures and $C_{N_0}^\infty$ such that the
extreme rays correspond to the ergodic measures.
\end{proof}

\begin{remark} We note that minimal dynamical systems have no infinite
invariant measures that take a finite value on a clopen set.
For an aperiodic dynamical system (and, in particular, for finite rank non-simple diagrams)
 such measures can occur, see \cite{bezuglyi_kwiatkowski_medynets_solomyak:2010}.
\end{remark}


\subsection{Vershik map}

By definition, a Bratteli diagram $B =(V,E)$ is called {\it ordered}
if every set $r^{-1}(v)$, $v\in \bigcup_{n\ge 1} V_n$, is linearly
ordered, see \cite{herman_putnam_skau:1992}. Thus, any two paths
from $E(v_0,v)$ are comparable with respect to the lexicographical
order. We call a finite or infinite path $e=(e_i)$ {\it maximal
(minimal)} if every $e_i$ is maximal (minimal) amongst the edges
from $r^{-1}(r(e_i))$. Notice that, for $v\in V_i,\ i\ge 1$, the
minimal and maximal (finite) paths in $E(v_0,v)$ are unique. Denote
by $X_{\max}$ and $X_{\min}$ the sets of all maximal and minimal
infinite paths from $X_B$, respectively. It is not hard to see that
$X_{\max}$ and $X_{\min}$ are finite sets for finite rank Bratteli diagrams
(Proposition 6.2 in \cite{bezuglyi_kwiatkowski_medynets:2009}). Let $X_B^*$
be the $\mathcal E$-invariant set of all infinite paths which are cofinal
neither to a maximal path nor to a minimal one. Then the set
$X_B\setminus X_B^*$ is at most countable for any finite rank diagram.

\begin{definition} Define a map $T:X_B^*\rightarrow X_B^*$ by
setting $$T(x_1,x_2,\ldots)=(x_1^0,\ldots,x_{k-1}^0,\overline
{x_k},x_{k+1},x_{k+2},\ldots),$$ where $k=\min\{n\geq 1 : x_n\mbox{
is not maximal}\}$, $\overline{x_k}$ is the successor of $x_k$ in
$r^{-1}(r(x_k))$, and $(x_1^0,\ldots,x_{k-1}^0)$ is the minimal path
in $E(v_0,s(\overline{x_k}))$. In this paper, we will refer to the
map $T$ as the {\it Vershik map} on the ordered Bratteli diagram $B$.
\end{definition}

\begin{remark}
(i) It is still unknown under what conditions the Vershik map can be
extended to a homeomorphism of $X_B$ for non-simple Bratteli diagrams.
We note only that it is not always possible \cite{medynets:2006}.

(ii) Since all orbits of $T$ coincide with classes of $\mathcal E$,
maybe except for at most countable collection of orbits, any
$\mathcal E$-invariant measure is also $T$-invariant and vice
versa.

\end{remark}

Throughout the paper we will always assume that each Bratteli
diagram of finite rank meets the following conditions:

\begin{itemize}

\item[(i)] The path space $X_B$ has
no isolated points, i.e., $X_B$ is a Cantor set.

\item[(ii)] The diagram
has the same number of vertices at each level, say $d$. So, each
incidence matrix is a $d \times d$ matrix.

\item[(iii)] The diagram has simple
edges between the top vertex $v_0$ and the vertices of the first
level, i.e., the vector $F_0$ consists of $1$'s. (This assumption is
not restrictive because any diagram can be isomorphically
transformed into a diagram with simple edges on the first level,
as in \cite[Lemma 9]{durand_host_scau:1999}.)

\item[(iv)] The cofinal equivalence relation is aperiodic, i.e. it
has no finite classes. This assumption is needed to exclude atomic
invariant measures from consideration.
\end{itemize}

%

\section{Structure of Invariant Measures}\label{SectionStructureInvariantMeasures}
In this section we describe the structure of the set of invariant measures. A key observation made here is that
ergodic measures occur as extensions of measures from simple pairwise disjoint subdiagrams
(Theorem \ref{TheoremGeneralStructureOfMeasures}). We begin
our study by describing the process of measure extension from a subdiagram,
 which is central for the paper.

Consider a Bratteli diagram $B =(V,E)$
where the vertex set $V = \bigcup_n V_n$ and the edge set $E= \bigcup E_n$ are as in Definition \ref{Definition_Bratteli_Diagram}.

\begin{definition} By a {\it subdiargam} of $B$, we mean a Bratteli diagram $S =
(W,R)$ constructed by taking some vertices at
each level $n$ of the diagram $B$ and then considering all the edges of $B$
that connect these vertices.
\end{definition}

\begin{remark} We notice that our definition of a subdiagram is not, in general, invariant under the telescoping, that is, the telescoping can add
additional edges not present originally.
\end{remark}

Let $S = (W,R)$ be a subdiagram of $B$. Consider the set $Y = Y_S$
of all infinite paths of the subdiagram $S$. Then the set $Y$ is
naturally seen as a subset of $X_B$. Let $\mu$ be a finite invariant
(with respect to the tail equivalence relation $\mathcal E$) measure
on $Y$. Let $X_S$ be the saturation
 of $Y$ with respect to $\mathcal E$. In other words, a
path $x\in X_B$ belongs to $X_S$ if it is $\mathcal E$-equivalent to
a path $y\in Y$. Then $X_S$ is $\mathcal E$-invariant and $Y$ is a
complete section for $\mathcal E$ on $X_S$. By the {\it extension of
measure $\mu$ to $X_S$} we mean the $\mathcal E$-invariant measure
$\widehat \mu$ on $X_S$ (finite or infinite) such that $\widehat
\mu$ induced on $Y$ coincides with $\mu$.

Although the procedure of the measure extension with respect to an equivalence relation
 is well-known, the geometric nature of the tail equivalence
relation makes this construction more illuminating.

Specifically, take a finite path $\ov e \in E_S(v_0, v)$ from the top vertex to a
vertex $v$ of level $n$ that belongs to the subdiagram $S$. Let $[\ov e]_S$ be the set of
all paths in $Y$ that coincide with $\ov e$ in the first $n$
edges. Then $[\ov e]_S$ is a cylinder subset of $Y$. For any finite
path $\ov e'$ from the diagram $B$ with the same range $v$ we set $\widehat\mu ([\ov e'])
 = \mu([\ov e]_S)$. In such a way, the measure $\widehat\mu$ is extended to the
 $\sigma$-algebra of Borel subsets of $X_B$ generated by all clopen sets of the form
 $[\ov x]$ where a finite path $\ov x$ has the range in a
vertex from $S$. Using the properties of tail equivalence relations,
one can show that such an extension is well-defined. Furthermore,
the {\it support} of $\widehat\mu$ is, by definition, the set $X_S$ of all paths which are cofinal
to paths from $Y$. We observe that $\widehat\mu(X_S)$ may be either finite or
infinite. In fact, one can use the following formula for computing $\widehat\mu(X_S)$.
 Let $W_n = W\cap V_n$ and set
 $X_S(n) = \{x = (x_i)\in X_B : r(x_i) \in W_i, \ \forall i \geq n\}$.
 Clearly, $X_S(n) \subset X_S(n+1)$. Then we have the following formula for the measure $\widehat\mu(X_S)$:
\begin{equation}\label{extension_method}
\widehat\mu(X_S) = \lim_{n\to\infty} \widehat\mu(X_S(n))
= \lim_{n\to\infty}\sum_{w\in W_n} \widehat h^{(n)}_w \mu([e_S(v_0,w)])
\end{equation}
where $\widehat h^{(n)}_w$ is the height of the tower $X^{(n)}_w$ in the diagram $B$ and $e_S(v_0, w)$
 is a finite path from $v_0$ to $w$ that belongs to $S$.

From now on, we may assume that a finite rank Bratteli diagram is
reduced by Theorem \ref{structure finite rank diagrams} to the
form (\ref{Frobenius Form: General}) when it is convenient for us.
Denote by $\Lambda$ the subset of $\{1,\ldots,m\}$ such that the
corresponding incidence matrices are non-zero in (\ref{Frobenius
Form: General}). For $\alpha\in \Lambda$, denote by $B_\alpha$ the
subdiagram of $B$ whose incidence matrices are $\{F_\alpha^{(n)}\}$.
The fact that the matrix $F_\alpha^{(n)}$ is strictly positive implies
that the subdiagram $B_\alpha$ is simple.

Let $Y_\alpha$ be the path space of the Bratteli diagram $B_\alpha,\
\alpha \in \Lambda$. Denote by $X_\alpha = \mathcal E(Y_\alpha)$
the saturation of $Y_\alpha$ with respect to the tail equivalence
relation. It is clear that $\{X_\alpha : \alpha \in \Lambda\}$ is a
partition of $X_B$ into Borel invariant subsets.

In the next theorem, we describe the structure of the supports of
ergodic invariant measures. The support of each ergodic measure turns out to
be the set of all paths that stabilize in some subdiagram, which
geometrically can be seen as ``vertical''. Furthermore, these
subdiagrams are pairwise disjoint for different ergodic measures.
Everywhere below the term ``measure'' stands for an $\mathcal E$-invariant
measure. Recall that by an infinite measure we mean any
$\sigma$-finite non-atomic measure which is finite (non-zero) on some clopen set.

%
\begin{theorem}\label{TheoremGeneralStructureOfMeasures} Let $B$ be a Bratteli diagram of finite rank.

(1) Each finite ergodic measure on $Y_\alpha$ extends to an
ergodic measure on $X_\alpha$. The extension can be a finite
or an infinite measure.

 (2) Each ergodic measure (both finite
and infinite) on $X_B$ is obtained as an extension of a finite
ergodic measure from some $Y_\alpha$.

(3) The number of finite and infinite (up to scalar multiple) ergodic measures is not
greater than $d$.

(4) We may telescope the diagram $B$ in such a way that for every
probability ergodic measure $\mu$ there exists a subset $W_\mu$ of
vertices from $\{1,\ldots,d\}$ such that the support of $\mu$ consists of
all infinite paths that eventually go along the vertices of $W_\mu$
only. Furthermore,

(4-i) $W_\mu\cap W_\nu = \emptyset$ for different ergodic measures
$\mu$ and $\nu$;

(4-ii) given a probability ergodic measure $\mu$, there exists a constant $\delta>0$
such that for any $v\in W_\mu$ and any level $n$
$$\mu(X_v^{(n)})\geq \delta
$$
where $X_v^{(n)}$ is the set of all paths that go through the vertex $v$ at level $n$;

(4-iii) the subdiagram generated by $W_\mu$ is simple and uniquely
ergodic. The only ergodic measure on the path space of the subdiagram
is the restriction of measure $\mu$.

(5) If a probability ergodic measure $\mu$ is the extension of a measure
 from the vertical subdiagram determined by a proper subset $W\subset\{1,\ldots,d\}$, then
$$\lim_{n\to\infty}\mu(X_v^{(n)}) = 0\mbox{ for all }v\notin W.$$
\end{theorem}

\begin{remark}
The recent paper \cite{bressaud_durand_maass:2009} contains a notion of a ``clean diagram'' for simple Bratteli-Vershik diagrams of finite rank, which
has some similarities with our description of the measure supports.
\end{remark}

\begin{proof} (I) Statements (1), (2), and (3) are similar to Lemma 4.2 from
\cite{bezuglyi_kwiatkowski_medynets_solomyak:2010} so that we
 give a sketch of the proof only.

Let $\mu$ be a finite or infinite ergodic measure on the
path-space $X_B$. Then there exists $\alpha$ such that $\mu$ is
supported on $X_\alpha$. As $Y_\alpha$ is a complete section of
$X_\alpha$, the restriction of $\mu$ to $Y_\alpha$ determines an
ergodic measure $\mu_0$ on $Y_\alpha$. Thus, to define a
measure on $X_\alpha$ we need to take any finite ergodic measure on
$Y_\alpha$ (due to Proposition \ref{PropositionSolomyakNote} we have
finitely many of them up to a normalization) and extend it by
invariance to $X_\alpha$. This process was described at the beginning of
this section, see equation (\ref{extension_method}).
We note that if the extended measure $\mu$ is infinite, but
finite on a clopen set, then the minimality of the tail equivalence
relation on $Y_\alpha$ implies that the restriction $\mu_0$ is a
finite measure. This proves (1), (2), and (3).

(II) To prove (4), we enumerate probability ergodic measures on
$X_B$ as $\mu_1,\ldots,\mu_p$. In view of (I), we may assume,
without loss of generality, that each measure $\mu_i$ is restricted
to a simple subdiagram $B_{\alpha_i}$. We start with the measure
$\mu_1$. Then
$$
\sum_v\limsup_{n\to\infty}\mu_1(X_v^{(n)})\geq \limsup_{n\to\infty}\sum_v\mu_1(X_v^{(n)}) = 1.
$$
Therefore, there exists a vertex $v_1$ with
$$\limsup_{n\to\infty}\mu_1(X_{v_1}^{(n)})=\delta_1>0.$$
This means that we can telescope the diagram so that $\mu_1(X_{v_1}^{(n)})>\delta_1/2$
for all levels $n$. Considering the set of vertices
$\{1,\ldots,d\}\setminus\{v_1\}$, choose a vertex $v_2$ (if
possible) such that for some positive number $\delta_2$
$$
\limsup_{n\to\infty}\mu_1(X_{v_2}^{(n)})=\delta_2>0.
$$
Telescope the diagram so that $\mu_1(X_{v_2}^{(n)})>\delta_2/2$ for
all levels $n$. Repeating this procedure finitely many times, we
will end up with a set of vertices $W_1$ such that
$$
\mu_1(X_{v}^{(n)})>\delta>0
$$
for all levels $n$ and any vertex $v\in W_1$ (here $\delta =
\frac{1}{2}\min_i\delta_i$) and such that
$$
\limsup_n\mu_1(X_v^{(n)}) = 0\mbox{ for all }v\notin W_1.
$$

We will further telescope the diagram to ensure that
$$
\sum_{k=n}^\infty\mu_1(\bigsqcup_{v\notin
W_1}X_v^{(k)})<\frac{1}{n}\mbox{ for any }n.$$

Consider the set $S_1$ of all paths that eventually go only through
the vertices from $W_1$. We claim that the measure $\mu_1$ is
supported on $S_1$. Indeed, consider the set
$$
R_1 = X_B\setminus S_1 = \bigcap_{n\geq 1}\bigcup_{k\geq
n}\bigsqcup_{v\notin W_1}X_v^{(k)}.
$$
Then
$$\mu_1(R_1) = \lim_{n\to\infty}\mu_1(\bigcup_{k\geq n}\bigsqcup_{v\notin
W_1}X_v^{(k)})\leq
\lim_{n\to\infty}\sum_{k=n}^\infty\mu_1(\bigsqcup_{v\notin
W_1}X_v^{(k)}) = 0,
$$
which proves the claim.

As soon as $W_1$ is constructed, we may repeat the arguments above
to find the corresponding sets $W_2,\ldots,W_p$ for the rest of the
ergodic measures.

We claim that $W_i\cap W_j = \emptyset$ for all $i\neq j$. Assume
the converse, i.e. that there are two probability ergodic measures
$\mu$ and $\nu$ and a vertex $w$ such that
$$
\mu(X_w^{(n)})\geq
\gamma\mbox{ and }\nu(X_w^{(n)})\geq \gamma
$$
for all $n$, where
$\gamma = \frac{1}{2}\min(\delta(\mu),\delta(\nu))>0$.

Set $C = \bigcap_{k\geq 1}\bigcup_{n\geq k}X_w^{(n)}$. It follows
that $\mu(C)\geq \gamma$ and $\nu(C)\geq \gamma$. Note that $C$ is
exactly the set of all paths that visit the vertex $w$ infinitely
many times, which is an $\mathcal E$-invariant set.
By ergodicity of $\mu$ and $\nu$, we see that $\mu(C) = \nu(C) = 1$.

Since $\mu$ and $\nu$ are mutually singular as distinct ergodic
measures, the Radon-Nikodym derivative satisfies
$$\frac{d\mu}{d(\mu+\nu)}(x)\equiv 0$$ for $\nu$-a.e. $x\in X_B$.

For every $x\in X_B$, let $v_n(x)$ denote the vertex of level $n$
the path $x$ goes through. Set $[x]_n = \{y \in X_B : y_j = x_j,\
j=1,...,n\}$. We observe that $h_{v_n(x)}^{(n)}\mu([x]_n) =
\mu(X_{v_n(x)}^{(n)})$, where $h_{v_n(x)}^{(n)}$ is the number of
paths from the vertex $v_n(x)$ to the top vertex.

As $\nu(C) = 1$, we have that for $\nu$-a.e. $x\in C$

\begin{eqnarray*} 0 & = & \frac{d\mu}{d(\mu+\nu)}(x) \\
& = & \lim_{n\to\infty}\frac{\mu([x]_n)}{(\mu+\nu)([x]_n)} \\
 & =& \lim_{n\to\infty}\frac{h_{v_n(x)}^{(n)}\mu([x]_n)}{h_{v_n(x)}^{(n)}(\mu+\nu)([x]_n)} \\
& = & \lim_{\{n: v_n(x) = w \}}
\frac{\mu(X_w^{(n)})}{\mu(X_w^{(n)})+\nu(X_w^{(n)})}
\\
& \geq & \frac{\gamma}{2}>0,
\end{eqnarray*}
which is a contradiction. Thus, statements (4-i) and (4-ii) are
proved.

(III) For each ergodic measure $\mu\in\{\mu_1,\ldots,\mu_p\}$,
denote by $B_\mu$ the subdiagram generated by the vertices
$W_\mu$. We note that the diagram $B_\mu$ is a subdiagram of the
corresponding simple diagram $B_{\alpha_i}$. Thus, we can telescope
the original
diagram in such a way that there is at least one edge between any
pair of vertices of $W_\mu$ at consecutive levels. This will ensure
that $B_\mu$ is a simple subdiagram.

Assume now that the diagram $B_\mu$ admits another probability
ergodic measure, say $\nu$. Denote by $Y_\mu$ the path space of
$B_\mu$ Then by the arguments above the measure $\nu$ and the restriction of $\mu$ to the path-space of $B_\mu$
are extended from (proper) disjoint subdiagrams of $B_\mu$.
Hence, there is a vertex $w\in W_\mu$ such that $\limsup_n\mu(X_w^{(n)}) = 0$, which is a contradiction.

(IV) Assume now that Statement (5) does not hold. Then there is $v_0\notin W$ such that
 $\limsup_n \mu(X_{v_0}^{(n)})>0$. Then the set $C = \bigcap_{k\geq 1}\bigcup_{n\geq k}X_{v_0}^{(n)}$ has $\mu$-measure one. Since
the set $C$ consists of paths that visit the vertex $v_0$ infinitely many times, this contradicts to the construction of the extension.
 \end{proof}

Motivated by the definition of exact rank measure-preserving transformations \cite{ferenczi:1997}, we give the following definition.

\begin{definition}\label{DefinitionExact} We say that a Bratteli diagram of a finite rank is of {\it exact finite rank}
 if there is a finite invariant measure $\mu$ and a constant
$\delta >0$ such that after a telescoping
$\mu(X_v^{(n)})\geq \delta$ for all levels $n$ and vertices $v$.

\end{definition}

 As a corollary, we immediately get the following version of
 Boshernitzan's theorem \cite{boshernitzan:1992}.

 \begin{corollary}\label{CorollaryBoshernitzan} All Bratteli diagrams of exact finite rank are uniquely ergodic.
 \end{corollary}

Interestingly, the condition of Boshernitzan for symbolic systems has been used to prove uniform convergence in the multiplicative ergodic theorem,
with applications to the spectral properties of the corresponding Schr\"odinger operators \cite{damanik_lenz:2006}.

%
\section{Unique Ergodicity of Simple Diagrams}
\label{SectionUniqueErgodicity}

  In this section, we will use the machinery of Birkhoff
contraction coefficient to answer the question when a simple
Bratteli diagram is uniquely ergodic.
Most of the results in this section are not new, but they are scattered in the literature, often with terminology different from ours.
We provide some (short) proofs for the reader's convenience.

The Birkhoff contraction coefficient method is widely used in matrix
theory and theory of Markov chains as the way to understand
asymptotic behavior of nonnegative matrix products. The Birkhoff
coefficient shows how matrix products ``squeeze'' the
orthant of positive vectors. The first results in the area appeared
in Birkhoff's fundamental works \cite{Birkhoff:1957} and
\cite{Birkhoff:1967}. We refer the reader to the books
\cite{hartfiel:book:2002} and \cite{seneta:book:1981} where a
detailed exposition of the material as well as extensive reference
list are presented. For the reader's convenience we include some
results from \cite{hartfiel:book:2002}.

\begin{definition} For two positive vectors $x,y\in\mathbb R^d$
define the {\it projective metric (Hilbert metric)} as $$D(x,y) =
\ln\max_{i,j}\frac{x_iy_j}{x_jy_i} = \ln \frac{\max_i
\frac{x_i}{y_i}}{\min_j \frac{x_j}{y_j}}
$$
where $(x_i)$ and $(y_i)$ are entries of the vectors $x$ and $y$.
\end{definition}
Denote by $\Delta$ the set of all positive probability vectors of
$\mathbb R^d$. Note that $(\Delta,D)$ is a complete metric space
(Theorem 2.5 in \cite{hartfiel:book:2002}).

The next theorem says that all non-negative matrices act as (weak)
contractions on the orthant of positive vectors. For the proof, see
Lemma 2.1 in \cite{hartfiel:book:2002}.
\begin{proposition}\label{TheoremMatrixContraction} Let $A$ be a
non-negative $d\times d$ matrix.
Then for any positive vectors $x,y\in \mathbb R^d$ we have
$D(Ax,Ay)\leq D(x,y)$.
\end{proposition}

\begin{definition} For a non-negative matrix $A$, we set $$\tau(A) = \sup_{x,y>0}\frac{D(Ax,Ay)}{D(x,y)}.$$
The coefficient $\tau(A)$ is called the {\it Birkhoff contraction
coefficient}.
\end{definition}

It follows from the definition that $D(Ax,Ay)\leq \tau(A)D(x,y)$.
Proposition \ref{TheoremMatrixContraction} implies that $0 \leq
\tau(A)\leq 1$. Note that the Birkhoff contraction coefficient has
the property $\tau(AB)\leq \tau(A)\tau(B)$.

\medskip
For a positive matrix $A = (a_{i,j})$, set
$$\phi(A) = \min_{i,j,r,s} \frac{a_{i,j}a_{r,s}}{a_{r,j}a_{i,s}}.
$$
If $A$ has a zero entry, then, by definition, we put $\phi(A)=0$.
The next theorem gives the formula for computing the Birkhoff
contraction coefficient.

\begin{proposition}[Theorem 2.6, \cite{hartfiel:book:2002}]\label{PropositionStrictContractability}
 Suppose that a matrix $A$ has a nonzero entry in each row.
Then $$\tau(A) = \frac{1-\sqrt{\phi(A)}}{1+\sqrt{\phi(A)}}.$$ In
particular, if $A$ is positive, then $\tau(A)<1$.
\end{proposition}

Let $\{A_k\}_{k\geq 1}$ be a sequence of $d\times d$ matrices.
Denote by $P_m^n$ the forward product $A_m A_{m+1}\cdots A_n,\ n>
m$.

\begin{definition} The products $P_m^n = (p_{i,j}^{(m,n)})$ are said to
tend to {\it row proportionality} if for all $k,s$ the sequence
$\frac{p^{(m,n)}_{k,i}}{p^{(m,n)}_{s,i} }$ converges (as
$n\to\infty$) to some constant $a=a(k,s,m)>0$ which does not depend
on the column index $i$.

Similarly, changing column indexes to row indexes, we can define the
notion of {\it column proportionality} (see \cite[Chapter
5]{hartfiel:book:2002} for details).
\end{definition}

\begin{remark}\label{remarkEqualityContractionForTranposes}
We note that if $P_m^n$ tends to row proportionality as
$n\to\infty$, then its transpose, which is the backward product of
$\{A_n^T\}$, tends to column proportionality. Proposition
 \ref{PropositionStrictContractability} also implies that
 $\tau(A_1\cdots A_n) = \tau(F_n\cdots F_1)$ where $F_i = A_i^T$.
\end{remark}

\begin{lemma}[Lemma 3.4, \cite{seneta:book:1981}]\label{LemmaRowProportionality} If $\{A_k\}$
is a sequence of
positive matrices, then $\tau(P_m^n)\to 0$ as $n\to \infty$ if and
only if the products $\{P_m^n\}$ tend to row proportionality.
\end{lemma}

\begin{definition} For any positive $d\times d$ matrix $A$
denote by $\Theta(A)$ the $D$-diameter (in the projective metric) of
the image of $\R^d_+$ under the action of $A$.
\end{definition}

The next lemma, which was proved by A. Fisher (see Proposition 6.13 and
Corollary 6.4 of \cite{fisherUniqueErgodicity:preprint}), is crucial for
our study.

\begin{lemma}\label{lemmaConeDiameter}
Let $A = (a_{i,j})$ be a positive matrix.
Then $\Theta(A) = \Theta(A^T)$. Furthermore, $$\Theta(A) =
\max_{i,j,k,l}\log \frac{a_{i,k}a_{j,l}}{a_{j,k}a_{i,l}}.$$
\end{lemma}

As a corollary of this result we deduce the following simple fact
saying that the image of the cone of positive vectors under $P_m^n$
has sufficiently small diameter in the projective metric $D$ when
$n$ is large enough if and only if the Birkhoff contraction
coefficient of $P_m^n$ tends to zero, as $n\to \infty$.

%
%
\begin{lemma}\label{LemmaConeContraction} Suppose that all matrices $\{A_k\}_{k\geq 1}$ are
positive. Then $\tau(P_m^n)\to 0$ as $n\to\infty$ if and only if for
given $\e> 0,\ m\in \N$, and any non-negative vectors $x,y$ there
exists $N\in \N$ such that $D(P_m^nx,P_m^ny)<\e$ for $n\geq N$.
\end{lemma}

\begin{proof} Set $F_k = A_k^T$. Suppose that $\tau(P_m^n)\to
0$ as $n\to \infty$. It follows from Remark
\ref{remarkEqualityContractionForTranposes} that $\tau(F_n\cdots
F_m)\to 0 $. Hence the backward product $(P_m^n)^T = F_n\cdots F_m$
tends to column proportionality.

Denote by $e_i$ the $i$-th column vector from the standard basis.
Consider $x = \sum x_ie_i$ and $y = \sum y_j e_j$ where the
summation is over indices with $x_i>0$ and $y_j>0$, respectively.
Then we get that
$$
D((P_m^n)^Tx,(P_m^n)^Ty)\leq \sum_{i,j=1}^d
D((P_m^n)^Te_i,(P_m^n)^Te_j).
$$
Thus it suffices to estimate the distance between the images of
basis vectors.
 Set $v_n = (P_m^n)^Te_i$ and $w_n = (P_m^n)^Te_j$. Then $v_n$ and $w_n$ are exactly the $i$-th and $j$-th columns of the matrix
$(P_m^n)^T$. Using the definition of projective metric $D$ and the
property of column proportionality of $(P_m^n)^T$, we get
 that $D(v_n,w_n)\to 0$ as $n\to\infty$. Thus, we obtain that
  $\Theta(P_m^n) = \Theta((P_m^n)^T)\to 0$.

Conversely, using the equality $\Theta(P_m^n) = \Theta((P_m^n)^T)$,
we get that $D(v_n,w_n)\to 0$ where $v_n$ and $w_n$ are the columns
of $(P_m^n)^T$. It follows from the definition of the metric $D$
that
$$
\frac{v_n(i)}{w_n(i)} \cdot \frac{w_n(j)}{v_n(j)} \to 1\mbox{ for
all }i,j.
$$
This implies precisely that the matrices $\{(P_m^n)^T\}$ tend to
column proportionality as $n\to\infty$. \end{proof}

Appropriate matrix norms may serve as numerical characteristics of
growth rate for matrix products. For a vector $v\in\mathbb R^d$
denote by $||v||_1$ the norm given by
$$
 ||v||_1 = \sum_i|v_i|.
$$
Similarly, for a square matrix $A = (a_{i,j})_{i,j}$ we denote by
 $||A||_1$ the entrywise 1-norm
\be \label{eq-norm} ||A||_1 = \sum_{i,j}|a_{i,j}|. \ee Note that
this is not the operator norm arising from the vector 1-norm.
However, it is easy to check that
$$
||AB||_1 \le ||A||_1 ||B||_1\ \ \mbox{and}\ \ ||Ax||_1\le ||A||_1
||x||_1,
$$
whenever the products are defined. Note also that \be
\label{eq-norm2} ||A||_1 = ||A{\ov 1}||_1 = ||{\ov 1}^T A||_1 \ee
for any non-negative matrix $A$, where $\ov 1 = (1,\ldots, 1)^T$.

\medskip Now we are ready to give the criterion of unique
ergodicity for a simple Bratteli diagram in terms of Birkhoff
contraction coefficients. In fact, the statement of Theorem \ref{theoremUniqueErgodicityInTermsTau} and a part of Proposition \ref{PropositionSufficientConditionsUniqueErgodicity} were proved in \cite{handelman:1999} even in a more general setting.
Also a version of this result
was earlier established by Fisher \cite[Theorem
1.3]{fisherUniqueErgodicity:preprint}, but with somewhat different
terminology and approach.

\begin{theorem}\label{theoremUniqueErgodicityInTermsTau}
Let $B$ be a simple Bratteli diagram of finite rank
with incidence matrices $\{F_n\}_{n\geq 1}$. Let $A_n = F_n^T$.
Then the diagram $B$ is uniquely ergodic if and only if $$\lim_{n\to
\infty}\tau(A_m\ldots A_n) = 0\mbox{ for every }m.$$
\end{theorem}
\begin{proof} Set $P_m^n = A_m\ldots A_n$. Denote the cone
$\bigcap_{n\geq m}P_m^n\mathbb R^d_+$ by $C_m$. By the compactness
argument, $C_m\neq\emptyset$. Furthermore, $A_{m}C_{m+1}=C_{m}$.
Therefore, for any vector $p^{(1)}\in C_1$ there exists a sequence
of nonnegative vectors $\{p^{(m)}\}_{m\geq 1}$ such that
$A_{m-1}p^{(m)}=p^{(m-1)}$. Such a sequence of vectors defines a
finite invariant measure. The converse is also true. It follows from
 Remark \ref{RemarkSimplexFirstLevelUnique} that in order to establish the
unique ergodicity, it is necessary and sufficient to show that $C_1$
is a single ray. Now the result follows immediately from Lemma
\ref{LemmaConeContraction}.
\end{proof}
\medskip

In the next proposition we collect a number of conditions yielding unique ergodicity
that can be easily checked in practice. For the proof, see Corollary
5.1 in \cite{hartfiel:book:2002} and Theorem 3.2 in
\cite{seneta:book:1981}.


\begin{proposition}\label{PropositionSufficientConditionsUniqueErgodicity}
Let $\{A_n\}_{n\geq 1}$ be transposes of primitive incidence matrices of a finite rank diagram $B$.

(1) The diagram $B$ admits a unique invariant probability measure on
$X_B$ if and only if there exists a strictly increasing sequence
$\{n_s\}$ such that
$$
\sum_{s=1}^\infty \sqrt{\phi(P_{n_s}^{n_{s+1}})} = \infty
$$
where $P_{n_s}^{n_{s+1}} = A_{n_s}\cdots A_{n_{s+1}}$. In
particular, if $$\sum_{n = 1}^\infty \sqrt{\phi(A_n)} = \infty,$$
then $B$ admits a unique invariant probability measure.

(2) If
$$
\sum_{n = 1}^\infty \left(\frac{m_n}{M_n} \right)=\infty,
$$
where $m_n$ and $M_n$ are the smallest and the largest entry of
$A_n$ respectively, then $B$ admits a unique invariant probability
measure.
\end{proposition}

\begin{example} Let $B$ be a simple Bratteli diagram with incidence matrices
$$
F_n = \left(
  \begin{array}{cccc}
    f_1^{(n)} & 1 & \cdots & 1 \\
    1 & f_2^{(n)} & \cdots & 1 \\
    \vdots & \vdots & \ddots & \vdots \\
    1 & 1 & \cdots & f_d^{(n)} \\
      \end{array}
\right).
$$
Let $q_n = \mbox{max}\{f_i^{(n)}f_j^{(n)} : i\neq j \}$. Compute
$\phi(F_n) = q_n^{-1}$. For $A_n = F_n^T$, we observe that if
$$
\sum_{n=1}^\infty \sqrt{\phi(A_n)} = \sum_{n=1}^\infty
\frac{1}{\sqrt q_n} =\infty,
$$
then there is a unique invariant probability measure on $B$. This
example generalizes an example considered in
\cite{ferenczi_fisher_talet} for the case of $2\times 2$ matrices.
\end{example}

As a corollary of Proposition
\ref{PropositionSufficientConditionsUniqueErgodicity} we immediately
obtain that if the incidence matrices do not grow too fast, then the
diagram admits a unique invariant measure.

\begin{corollary}\label{CorollarySufficientConditionUniqueErgodicity}
  If a simple Bratteli diagram with incidence matrices $\{F_n\}_{n\geq 1}$ satisfies the
condition $||F_n||_1\leq Cn$ for some $C>0$ and all sufficiently
large $n$, then the diagram admits a unique invariant probability
measure.

In particular, this result holds if the diagram has only finitely
many different incidence matrices.
\end{corollary}
\begin{proof} Denote by $m_n$ and $M_n$ the smallest and the largest
entry of $F_n$ respectively. Using the simplicity of the diagram
and an appropriate telescoping, we may assume that $m_n\ge 1$ for
all $n$. By the definition of the entrywise matrix 1-norm, we get
that $$\frac{m_n}{M_n}\geq \frac{1}{||F_n||_1} \geq \frac{1}{Cn}$$
for all $n$ large enough. The result follows from Proposition
\ref{PropositionSufficientConditionsUniqueErgodicity}.
\end{proof}

\begin{remark} (1) This corollary gives another proof of the fact
that linearly recurrent systems are uniquely ergodic, which was
originally established in Proposition 5 of
\cite{cortez_durand_host_maass:2003}.

(2) It is mentioned in \cite[p. 528]{hajnal:1976} that the
products of the following sets of positive matrices tend to column
proportionality and, in particular, give rise to uniquely ergodic
systems:
\begin{enumerate}\item[(i)] Any set of primitive incidence matrices
which pairwise commute.
\item[(ii)] The set $\Sigma$ of primitive incidence matrices such that if $A\in \Sigma$ and $F$ is primitive,
then $AF$ and $FA$ are primitive.
\end{enumerate}
\end{remark}

In the next example we show how the technique of Bratteli diagrams can
be used to derive a sufficient condition of unique ergodicity for
 generalized Morse sequences. See the papers \cite{keane:1968} and \cite{martin:1977} for more information about these systems
 and a complete characterization of unique ergodicity.

\begin{example}\label{ExampleGeneralizedMorseSequence} Let $G$ be a finite abelian group with group operation $+$.
Each element $g\in G$ acts on finite words $a = a_0\ldots a_p$, $a_i\in G$, by $\sigma_g(a)[i] = a_i+g$, $i=0,\ldots,p$.
For two finite words $a = a_0\ldots a_p$ and $c = c_0\ldots c_q$ over $G$, we define
 $$
a\times c = \sigma_{c_0}(a)\sigma_{c_1}(a)\ldots \sigma_{c_q}(a)\mbox{ (concatenation of words)}.
$$

Let $\{b^{(n)}\}_{n\geq 1}$ be a sequence of finite words over $G$. We assume
that the first letter $b^{(n)}[0] = 0$ (group identity), $|b^{(n)}|\geq 2$, and
all elements from $G$ occur in every $b^{(n)}$. Define the infinite sequence $$
\omega = b^{(1)}\times b^{(2)}\times \cdots.
$$
Consider the symbolic dynamical system $(X, T)$ generated by the shift $T$ on the closure $X$ of $T$-orbit of $\omega$. Points from $X$ are represented by bi-infinite sequences. Then $(X,T)$ is called a {\it generalized Morse dynamical system}. The classical Morse system is included in this scheme.

Denote by $\fr(g,b^{(n)})$ the frequency of an element $g\in G$ in the word $b^{(n)}$. The following fact is ``folklore''
and was originally established by methods of symbolic dynamics.
\medskip

\noindent
{\bf Claim:} {\it If \begin{equation}\label{EquationGeneralizeMorseSequence}\sum_{n\geq 1}\min_{g\in G}\fr(g,b^{(n)}) = \infty,\end{equation}
 then the system is uniquely ergodic.}
\medskip

In fact, this is a necessary and sufficient condition for unique ergodicity when $G$ has two elements \cite{keane:1968};
the criterion of \cite{martin:1977} is stated in different terms.

This result can be proved by using the following approach:
find a Bratteli-Vershik model for $(X,T)$ and then show that
condition (\ref{EquationGeneralizeMorseSequence})
 allows us to apply Proposition \ref{PropositionSufficientConditionsUniqueErgodicity}.

Denote by $\lambda_n$ the length of the word $c^{(n)} = b^{(1)}
\times\cdots\times b^{(n)}$. For each $g\in G$,
set
$$B_n(g) = \{x\in X : x[0,\lambda_n-1] = \sigma_g(c^{(n)})\}.$$
Then the sets
$X_g^{(n)} = \{B_g(n),\ldots T^{\lambda_n-1}B_g(n)\}$, $g\in G$,
are disjoint; and $\Xi_n = \{X_g^{(n)} : g\in G\}$ form a Kakutani-Rokhlin
partition of $X$,
see \cite{martin:1977} for the details. Furthermore, one can check that the sequence $\{\Xi_n\}_{n\geq 1}$ is nested. Thus,
we can use the sequence $\{\Xi_n\}_{n\geq 1}$ to construct an ordered finite rank Bratteli diagram $B$.

Note that the ordering on the diagram $B$ has only finitely many maximal and minimal paths. Denote by $X_B^0$ a (countable) set of paths which
 are cofinal either to a minimal path or to a maximal one.
 Thus, the Vershik map $T_B$ determined by the ordering is well-defined and continuous on
$X_B' = X_B\setminus X_B^0$.

 Set
 $X_0 = \bigcup_k T^k(\bigcap_n\bigcup_g B_g(n))$ and $X'=X\setminus X_0$.
 Using the finiteness of $G$, one can show that $X_0$ is a countable set. Furthermore, the dynamical systems $(X_B',T_B)$ and $(X',T)$ are
(Borel) isomorphic and share the same set of invariant measures.
Therefore, in order to check the unique ergodicity of the generalized Morse system, it is enough to do this for $(X_B',T_B)$.

Since each tower in a Kakutani-Rokhlin partition $\Xi_n$ is exactly defined by an element $g\in G$,
there is a natural correspondence between vertices of level $n$ in the diagram $B$ and elements of $G$.
Using properties of $\{\Xi_n\}_{n\geq 1}$, one can check that that
the $n$-th incidence matrix of the diagram $B$ is as $F_n=(f_{g,h}^{(n)})$, where $f_{g,h}^{(n)}$ is the number of occurrences of
$h$ in the word $\sigma_g(b^{(n)})$. We observe that $f^{(n)}_{g,h} = f^{(n)}_{0,g-h}$ for any $g,h\in G$. Hence,
$$
m_n = \min_{g,h\in G}f^{(n)}_{g,h} = \min_{q\in G}f^{(n)}_{0,q} = \min_{g\in G}\fr(g,b^{(n)}) |b^{(n)}|,
$$ where $|b^{(n)}|$ is the length of $b^{(n)}$.
Observe that each row in the matrix $F_n$ sums up to $|b^{(n)}|$. Hence $M_n = \max_{g,h}f^{(n)}_{g,h} < |b^{(n)}|$. Thus, we conclude that
$$\min_{g\in G}\fr(g,b^{(n)}) = \frac{m_n}{|b^{(n)}|}\leq \frac{m_n}{M_n}.$$ Now the claim follows from
 Proposition \ref{PropositionSufficientConditionsUniqueErgodicity}.

We observe that by refining the partitions $\{\Xi_n\}_{n\geq 1}$ one can construct a topological (finite rank) Bratteli-Vershik
 model for $(X,T)$.

\end{example}

%
%

%
%

\section{Quantitative Analysis of Measures}\label{SectionQuantitativeAnalysis}

Throughout this section, we assume that the Bratteli diagram $B$ is simple and uniquely ergodic.
Our goal in the section is to study the asymptotic behavior of tower heights and of measures of tower bases. Since the heights of
towers determine the recurrence time for
points from the bases of towers, our study can be viewed as an ``adic version'' of the quantitative recurrence analysis.

We start by translating the ergodic theorem into the language of Bratteli diagrams.
Let $B$ be a simple Bratteli diagram of finite rank with a unique ergodic probability measure $\mu$. Without loss of generality
(after telescoping) we
can assume that all vertices of consecutive levels of $B$ are connected by an edge. Then it is easy to
enumerate the edges of the Bratteli diagram so that this
ordering defines a continuous Vershik map $T = T_B$
(see Section 3
of \cite{herman_putnam_skau:1992} for details).

Fix an integer $m>0$. For each infinite path $x\in X_B$, denote by
$v_m(x)$ the vertex of level $m$ the path $x$ goes through. Denote
also by $e(v_0, v_m(x))$ the finite segment of the path $x$ between the vertices $v_0$ and $v_m(x)$.
 Let $i_m(x)$ be the least integer such that $T^{-i_m(x)}$ maps $e(v_0, v_m(x))$ to the minimal
 finite path from the set $E(v_0,v_m(x))$. Similarly, let $j_m(x)$ be the least
integer such that $T^{j_m(x)}$ maps $e(v_0, v_m(x))$ to the maximal
path from $E(v_0,v_m(x))$. Notice that $h_{v_m(x)}^{(m)} = j_m(x) +
i_m(x)$.

Then, by the pointwise ergodic theorem and unique ergodicity of
$(X_B,T_B)$, we get that

\begin{equation}\label{pointwise_ergodic_theorem}
\mu(B_n(w)) = \lim\limits_{m\to\infty}\frac{1}{i_m(x)+j_m(x)}\sum\limits_{i = -i_m(x)}^{j_m(x)} 1_{B_n(w)}(T^i(x))
\end{equation}
for every $x\in X_B$.

The sum in the right-hand side of (\ref{pointwise_ergodic_theorem}) is equal to the number of paths that connect the vertex $w$ of level $n$ to the vertex $v_m(x)$ of level $m$. Hence, we obtain the following result.

\begin{proposition}\label{PropositionErgodicTheorem}
Let $B$ be a simple uniquely ergodic Bratteli diagram, and let  $\mu$ be the unique invariant probability measure on
$X_B$. Then for any vertices $v$, $w$, and any level $n$, we have
$$
\mu(B_n(w)) = \lim\limits_{m\to\infty}\frac{(F_{m-1}\cdots
F_n)_{v,w}}{h_{v}^{(m)}}.
$$
\end{proposition}

\begin{remark}
We should note that such an interpretation of the pointwise ergodic theorem first appeared in \cite[Theorem 2]{vershik_kerov:1981}, see also \cite[Lemma 3.4]{mela:2006}.
\end{remark}

\begin{lemma}\label{LemmaHeightGrowthUniqueErgodicity} Let $B$ be a simple uniquely ergodic Bratteli diagram of finite rank.
The diagram $B$ can be telescoped to a new diagram with incidence matrices $\{F_n\}_{n\geq 1}$ such that the following properties hold:

(i) there exist a non-negative probability vector $\xi$ and strictly positive vectors $\{\eta^{(n)}\}_{n\geq 1}$ such that for any $n>0$
and any vector $x\in \mathbb R^d_+$ we have
$$
\lim\limits_{m\to\infty}\frac{F_m\cdots F_nx}{||F_m\cdots
F_n x||_1}=\xi
$$
and
$$
\lim\limits_{m\to\infty}\frac{x^TF_m\cdots F_n}{||x^TF_m\cdots
F_n||_1}=(\eta^{(n)})^T> 0;
$$

(ii)
 $$
\frac{(\eta^{(n+1)})^TF_n}{||(\eta^{(n+1)})^T F_n||_1} =
(\eta^{(n)})^T
$$
and $\eta^{(n)}\to \eta\geq 0\mbox{ as }n\to\infty.$
\end{lemma}

\begin{proof} (i) Denote by $\{F_n\}_{n\geq 1}$ the incidence matrices of the diagram $B$. Since $B$ is uniquely ergodic,
we obtain, by Theorem \ref{theoremUniqueErgodicityInTermsTau}, that
$\tau(F_n^T\cdots F_m^T)\to 0$ for any fixed $n$ as $m\to\infty$. Applying Lemmas
\ref{lemmaConeDiameter} and \ref{LemmaConeContraction}, we conclude
that the $D$-diameter of the cone $C_m^{(n)} = F_m\cdots F_n\mathbb
R_+^d$ tends to zero as $m\to\infty$. Hence, by compactness of the
simplex of probability vectors, there exists a non-negative
probability vector $\xi^{(1)}$ and a subsequence $\{m_k\}_{k\geq 1}$
such that $C_{m_k}^{(1)}\to ray(\xi^{(1)})$ as $k\to\infty$. Telescope
the diagram along the sequence $\{m_k\}_{k\geq 1}$. For convenience, we denote the new incidence matrices by the same
symbols $\{F_n\}_{n\geq 1}$.

Applying the same arguments for $n=2,3,\ldots$, we inductively
telescope the diagram to new levels and find non-negative
probability vectors $\{\xi^{(n)}\}_{n\geq 1}$ such that $C_m^{(n)} \to ray(\xi^{(n)})$ as $m\to\infty$ for every fixed $n$.

It follows from the construction that for any non-negative vector
$x$
\begin{equation}\label{eq-rowProportionalityConvergence}
 \frac{F_m\cdots F_n x}{||F_m\cdots F_nx||_1} \to \xi^{(n)}\mbox{ as}\ m\to\infty.
\end{equation}
Setting $x = F_{n-1}y$ for some non-negative vector $y$, we see that
$$
\frac{F_m\cdots F_nF_{n-1} y}{||F_m\cdots F_nF_{n-1}y||_1} \to \xi^{(n-1)}\mbox{ as }m\to\infty.
$$
Hence $\xi^{(n)}=\xi^{(n-1)}=\ldots = \xi^{(1)} = \xi$.

(ii) To show the existence of a probability vector $\eta^{(n)}$ that
satisfies the condition of the lemma, we consider a decreasing sequence of cones $\widetilde{C}_m^{(n)} = F_n^T\ldots F_m^T\mathbb
R_+^d$. Lemma \ref{lemmaConeDiameter} implies that the $D$-diameter of these cones tends
to zero as $m\to\infty$. It follows that $\bigcap_{m\geq n} \widetilde{C}_m^{(n)} = ray(\eta^{(n)})$ for some
strictly positive probability vector $\eta^{(n)}$. Clearly, we can further telescope the diagram
to ensure that $\eta^{(n)}\to\eta$ for some probability non-negative vector $\eta$.
Verifying condition (ii) of the lemma is straightforward.
 \end{proof}

 \begin{remark} (1) If a Bratteli diagram $B$ is stationary, i.e.
$F_n=F$ for all $n$,
 then $\xi$ and $\eta$ are the normalized right and left Perron-Frobenius eigenvectors of $F$, respectively.

(2) In general, vectors $\xi$ and $\eta$ may have zero coordinates (see Example \ref{ExampleUniqueErgodicZeroCoordinates}).

(3) The entries of vectors $\{\eta^{(n)}\}$ represent (non-normalized) values of the invariant measure on cylinder sets. Set
$$
p^{(n)} = \frac{\eta^{(n)}}{||(\eta^{(2)})^T F_1||_1\cdots ||(\eta^{(n)})^T F_{n-1}||_1}.
$$
It follows from Lemma \ref{LemmaHeightGrowthUniqueErgodicity} that $F_n^T p^{(n +1)} = p^{(n)}$ for all $n$. Thus, this sequence defines
a probability measure (see Theorem \ref{Theorem_measures_general_case}).
 \end{remark}

Next we explore these questions for Bratteli diagrams of exact finite rank, see Definition~\ref{DefinitionExact}. Recall that such diagrams
are all uniquely ergodic (Corollary~\ref{CorollaryBoshernitzan}).

\begin{definition} For any two sequences of real numbers $\{x_n\}$ and $\{y_n\}$,
we will write $x_n\sim y_n$ as $n\to\infty$ to indicate that
$\lim_{n\to\infty} x_n/y_n = 1$.
\end{definition}

 The following simple proposition
shows that if the measures of tower bases have the same asymptotic growth, then so do the heights of towers.

\begin{proposition}\label{PropositionExactHeightGrowthEquivMeasureGrowth}
Let $B$ be a simple Bratteli diagram of exact finite rank with the probability invariant measure $\mu$.

(1) Then
$$\inf_{v,w,n} \frac{\mu(B_n(v))}{\mu(B_n(w))}>0\mbox{ if and only if }\inf_{v,w,n}\frac{h^{(n)}_v}{h^{(n)}_w} >0$$
where $B_n(w)$ is the base of the tower $X^{(n)}_w$
 and $h^{(n)} = F_{n-1}\cdots F_1 \overline 1$ is the vector representing the tower heights.

(2) If either condition holds, then the vector $\xi$ found in Lemma \ref{LemmaHeightGrowthUniqueErgodicity} is strictly positive and
(after an appropriate telescoping)
$$
h_w^{(n)}\sim \xi_w ||F_{n-1}\cdots F_1||_1
$$
and
 $$
 \mu(B_n(w)) \sim \frac{\rho_w}{||F_{n-1}\cdots F_1||_1}.
 $$
for some strictly positive vector $\rho= (\rho_w)$.

\end{proposition}
\begin{proof} Both statements are immediate from Lemma \ref{LemmaHeightGrowthUniqueErgodicity}
and the fact that $0< \delta\leq \mu(X_v^{(n)}) = h_v^{(n)}\mu(B_n(v)) \leq 1$ for all $v$ and $n$.
\end{proof}

The following proposition defines a large class of diagrams of exact finite rank whose towers grow with the same speed.
We note that the condition used in the next proposition is sometimes referred to as the ``compactness'' condition.

\begin{proposition}\label{PropositionSufficientCondBoundedAway} Let $B$ be a simple Bratteli finite rank diagram with the incidence matrices
$\{F_n\}_{n\geq 1}$. Suppose that there is a constant $c>0$ such that $m_n/M_n\geq c$, for all $n$, where $m_n$ and $M_n$ are the smallest
 and the largest entry of $F_n$, respectively. Then:

(1) the diagram $B$ is of exact finite rank;

(2) $h_w^{(n)}/h^{(n)}_v\geq c$ for all levels $n$ and all vertices $v$ and $w$.
\end{proposition}

\begin{proof} The unique ergodicity follows from Proposition \ref{PropositionSufficientConditionsUniqueErgodicity}.
Denote by $f^{(n,m)}_{i,j}$ the entries of the product matrix $F_m\cdots F_n$. The entries of $F_n$ are denoted by $f^{(n)}_{i,j}$.
 We claim that
$f^{(n,m)}_{p,v}/f^{(n,m)}_{p,w}\geq c$ for every $m\geq n$ and all vertices $p,v,w$. By induction, we need to show that if this inequality holds for $m$, then it is true for $m+1$. Indeed,
\begin{eqnarray*}\frac{f^{(n,m+1)}_{p,v}}{f^{(n,m+1)}_{p,w}} & = &
\frac{\sum_r f^{(m+1)}_{p,r} f^{(n,m)}_{r,v} }{\sum_i f^{(m+1)}_{p,i} f^{(n,m)}_{i,w}} \\
& = &
\sum_r \frac{f^{(m+1)}_{p,r} f^{(n,m)}_{r,w}}{\sum_i f^{(m+1)}_{p,i} f^{(n,m)}_{i,w}} \frac{f^{(n,m)}_{r,v}}{f^{(n,m)}_{r,w}}
\\
& \geq & c \sum_r \frac{f^{(m+1)}_{p,r} f^{(n,m)}_{r,w}}{\sum_i f^{(m+1)}_{p,i} f^{(n,m)}_{i,w}}\\
& = & c.
\end{eqnarray*}
It follows from Proposition~\ref{PropositionErgodicTheorem} that
$$\frac{\mu(B_n(v))}{\mu(B_n(w))} = \frac{f^{(n,m)}_{p,v} h_p^{(m)}}{f^{(n,m)}_{p,w} h_p^{(m)}} = \lim\limits_{m\to\infty} f^{(n,m)}_{p,v}/f^{(n,m)}_{p,w}\geq c $$ for all $v,w$.
Note also that $$h_w^{(n)}= \sum_r f^{(n)}_{v,r}h_r^{(n-1)}\geq c\sum_rM_n h_r^{(n-1)}\geq c h_v^{(n)}$$ for all $w$ and $v$.
Therefore, $$\frac{\mu(X^{(n)}_v)}{\mu(X^{(n)}_w)} = \frac{h_v^{(n)}\mu(B_n(v))}{h_w^{(n)}\mu(B_n(w))} \geq c^2$$ for all levels $n$ and vertices $v,w$. This proves the proposition.
\end{proof}

%
The following example shows that there are diagrams of exact finite rank whose tower heights obey different asymptotic rates.

\begin{example}
Consider a simple finite rank Bratteli diagram $B$ determined by the sequence of incidence matrices
$$
F_n = \left (\begin{array}{cc}1 & 1 \\
n & 1 \end{array}\right).
$$
By Corollary \ref{CorollarySufficientConditionUniqueErgodicity}, this diagram is
uniquely ergodic. Denote by $c_n$ and $d_n$ the $(1,1)$- and $(2,1)$-entry of $F_n\cdots F_1$, respectively. By induction, one can show that
$$
F_n\cdots F_1 = \left (\begin{array}{cc}c_n & c_n \\
d_n & d_n \end{array}\right).
$$
Hence $h_1^{(n)} = 2 c_{n-1}$ and $h_2^{(n)} = 2 d_{n-1}$ for all $n$. Using the recurrence relations $c_n = c_{n-1} + d_{n-1}$ and
$d_n = nc_{n-1} + d_{n-1}$, we see that
\begin{equation}\label{eqRecurrenceCn}c_n = 2 c_{n-1} + (n-2) c_{n-2}\end{equation}
and
\begin{equation}\label{eqRecurrenceDn}d_n = (n+1)c_{n-1} + (n-2)c_{n-2}. \end{equation}

Denote by $H_n(z)$ the n-th Hermite polynomial, i.e. $H_0(z) = 1$, $H_1(z) = 2z$, and for all $n\geq 1$
\begin{equation}\label{eqRecurrenceHn}H_{n+1}(z) = 2zH_{n}(z) - 2nH_{n-1}(z).\end{equation}
It follows from (\ref{eqRecurrenceCn}) and (\ref{eqRecurrenceHn}) that
\begin{equation}\label{eqFormulaCn}c_n = \left(\frac{-i}{\sqrt 2}\right)^{n-1}H_{n-1}(i\sqrt 2).\end{equation}

The asymptotic formula for Hermite polynomials is given by
\begin{equation}\label{eqAsHermite}
H_n(z) = \sqrt{2}\, \exp(z^2/2) (2n/e)^{n/2}\cos[z\sqrt{2n+1} - (\pi n)/2](1+q_n(z)),
\end{equation}
where $z\in \mathbb C\setminus \mathbb R$ and $q_n(z) \to 0$, see \cite{rusev:1976}. It follows that
\begin{equation}\label{eqHermiteRatioAssympt}\frac{H_n(i\sqrt 2)}{H_{n-1}(i\sqrt 2)}\sim i\sqrt {2n}\mbox{ as }n\to\infty. \end{equation}

\noindent
{\it Claim 1.} $h_n^{(1)}/h_n^{(2)}\to 0$ when $n\to\infty$.
\medskip

Indeed, in view
of (\ref{eqRecurrenceCn}) and (\ref{eqRecurrenceDn}) it is enough to show that $c_{n-1}/c_n\to 0$ when $n\to\infty$, which immediately
follows from (\ref{eqFormulaCn}) and (\ref{eqHermiteRatioAssympt}).
\medskip

\noindent{\it Claim 2.} Let $\mu$ be the probability invariant measure on $X_B$. Then $\mu(X_1^{(m)}) \to 1/2$ when $m\to \infty$,
hence the diagram has exact finite rank.
\medskip

We will need the second (linearly independent of $H_n(z)$) solution of (\ref{eqRecurrenceHn}) given by
 $$Q_n(z) = - \int\limits_{-\infty}^\infty \frac{e^{-t^2}H_n(z)}{t-z}dt,\; z\in \mathbb C\setminus \mathbb R,$$
see \cite{rusev:1976} for the details.
The functions $Q_n(z)$ are called the Hermite functions of second kind. We note that any other solution of
(\ref{eqRecurrenceCn}) is a linear combination of $H_n(z)$ and $Q_n(z)$ \cite{rusev:1976}. (We are thankful to L.~Golinskii and P.~Nevai for their suggestions to use the functions $Q_n(z)$.)

 The following asymptotic formula was also established in
\cite{rusev:1976} for $z$ in the upper half-plane
\begin{equation}\label{eqAsSecondKindHermite}
Q_n(z) = (-i)^{n+1}\pi\sqrt 2 (2n/e)^{n/2} \exp[-z^2/2 + iz\sqrt{2n+1}](1+k_n(z)),
\end{equation}
where $k_n(z)\to 0$. It follows from (\ref{eqAsHermite}) and (\ref{eqAsSecondKindHermite}) that
 \begin{equation}\label{eqRationHermiteSecondKindHermite} \frac{Q_n(i\sqrt 2)}{H_n(i\sqrt 2)} \to 0\mbox{ as }n\to\infty.
\end{equation} Note also that $$\frac{Q_{n-1}(i\sqrt 2)}{Q_n(i \sqrt 2)} \sim \frac{i}{\sqrt {2n}}\mbox{ as }n\to\infty.$$

Applying the pointwise ergodic theorem (Proposition \ref{PropositionErgodicTheorem}), we get that
\begin{equation}\label{eqErgThApplied}\mu(X_1^{(m)}) = \lim\limits_{n\to\infty} \frac{h_1^{(m)}(F_{n-1}\cdots
F_m)_{1,1}}{h_{1}^{(n)}}.\end{equation}

Set $R_n^{(m)} = (F_{n-1}\cdots F_m)_{1,1}$. We observe that the sequence $\{R_n^{(m)}\}_{n\geq m}$
satisfies the recurrence relation (\ref{eqRecurrenceCn}) and initial conditions $R_m^{(m)} = 1$ and $R_{m+1}^{(m)} = m + 1$.
Note that $R_n^{(1)} = c_n$. Thus,

$$R_n^{(m)} = \left(\frac{-i}{\sqrt 2}\right)^{n-1}\left(\alpha_m H_{n-1}(i\sqrt 2) + \beta_m Q_{n-1}(i\sqrt 2) \right)
\mbox{ for all }n\geq m,$$ where the constants
$\alpha_m$ and $\beta_m$ are uniquely determined by the initial conditions. The asymptotic ratio
(\ref{eqRationHermiteSecondKindHermite}) implies that $$\frac{R_n^{(m)}}{h_1^{(n)}} \to \frac{\alpha_m}{2}\mbox{ as }n\to\infty.$$
It follows from (\ref{eqErgThApplied}) that $\mu(B_m(1)) = \alpha_m/2$ and $$\mu(X_1^{(m)}) = \left(\frac{-i}{\sqrt 2}\right)^{m-2}\alpha_m H_{m-2}(i\sqrt 2).$$

Solving the system of equations $R_m^{(m)} = 1$ and $R_{m+1}^{(m)}
= m + 1$ for $\alpha_m$ and $\beta_m$, we obtain that
$$\alpha_m = \left(\frac{-\sqrt 2}{i}\right)^m
\frac{(i/\sqrt 2) Q_m(i\sqrt 2) - (m+1)Q_{m-1}(i\sqrt 2)}%
{H_{m-1}(i\sqrt 2)Q_m(i\sqrt 2) - H_m(i\sqrt 2)Q_{m-1}(i\sqrt 2)}. $$

Now it is straightforward to check that $\mu(X_1^{(m)}) \to 1/2$ as $m\to\infty$. We skip the computation.
\end{example}

\begin{remark} We note that the uniform growth of tower heights does not guarantee the unique ergodicity of the diagram. As an example, consider
the Bratteli diagram $B$ with incidence matrices $$F_n = \left (\begin{array}{cc}n^2 & 1 \\
1 & n^2 \end{array}\right).$$ Setting $h^{(n)} = F_{n-1}\cdots F_1
\overline 1$, we note that $h^{(n)}_1= h^{(n)}_2 =
2^{-1}||F_{n-1}\cdots F_1||_1$. However, it was shown in
\cite[Proposition 3.1]{ferenczi_fisher_talet} that the diagram $B$
has exactly two finite ergodic invariant measures (see also more general Example \ref{ExampleExtensionOfMeasureFromOdometers}).
We may also apply the methods of Section \ref{SectionMeasureExtension} to show that each of these
measures is obtained as an extension of a unique invariant measure from the left (right) vertical subdiagram.
\end{remark}
\medskip

The following example presents a uniquely ergodic diagram of non-exact finite rank
with different growth of tower heights.

\begin{example}\label{ExampleUniqueErgodicZeroCoordinates} Consider the Bratteli diagram $B$ determined by the
incidence matrices $$F_n = \left (\begin{array}{cc}1 & 1 \\
1 & n \end{array}\right).$$ By Corollary
\ref{CorollarySufficientConditionUniqueErgodicity}, this diagram is
uniquely ergodic. However, the following result holds:
\medskip

\noindent
\textit{Claim}. The diagram $B$ is not of exact finite rank.
\medskip

Indeed, let $h_i^{(n)}$ be
the height of the $i$-th tower at level $n$, $i =1,2$. Clearly, $h_1^{(n)}\leq
h_2^{(n)}$. Hence $$\frac{h_1^{(n+1)} }{h_2^{(n+1)}} =
\frac{h_1^{(n)} + h_2^{(n)}}{ h_1^{(n)} +nh_2^{(n)}}\leq \frac{2
h_2^{(n)}}{nh_2^{(n)}}\to 0\mbox{ as }n\to \infty.$$

It follows that $$\frac{h_1^{(n)}}{h_2^{(n+1)}}\leq \frac{h^{(n)}_1}{nh_2^{(n)}}\leq \frac{2}{n(n-1)}.$$ Now if we
take the invariant probability measure $\mu$ on the right (vertical) subdiagram, then the convergence of
$\sum_n h^{(n)}_1/h_2^{(n+1)}$ and Proposition
\ref{finiteness_condition} below imply that the extension of $\mu$ is a finite invariant measure.
Thus, the unique invariant measure is the extension of $\mu$.
 Hence by Theorem \ref{TheoremGeneralStructureOfMeasures} we get that
$\mu(X_1^{(n)})/\mu(X_2^{(n)})\to 0$ as $n\to\infty$.
\end{example}

%
%
\section{Extension of Measures from Subdiagrams}\label{SectionMeasureExtension}
In view of the structural results of Section
\ref{SectionStructureInvariantMeasures},
each invariant measure on a finite rank Bratteli diagram
is obtained as an extension of a measure
from some subdiagram. In this section we further study this construction
by establishing some algebraic conditions for finiteness of the extension.
The motivation for this is to obtain some quantitative properties of
diagrams.

\subsection{General Condition}
Let $\ov W = \{W_n\}$ be a sequence of finite subsets of $V_n$. We will consider
the non-trivial case when $W_n$ is a proper subset of $V_n$ for all $n$. Denote $W'_n = V_n \setminus W_n$.
Thus, the sequence $\ov W$ determines a proper Bratteli subdiagram $B(\ov W)$ which is formed by the vertices
from $\ov W$ and the edges that connect them. Let $Y = Y_{B(\ov W)}$
 be the path space of $B(\ov W)$. The following proposition may be viewed as an analogue of the Kac lemma on the first
return map in measurable dynamics.

\begin{proposition}\label{finiteness_condition}
Let $B$ be a finite rank diagram with incidence matrices
$\{F_n = (f^{(n)}_{v,w})\}$, and $B(\ov W)$ is a subdiagram as above. Let $\mu$ be a finite
invariant measure on $B(\ov W)$.

(1) Suppose the extension $\wh \mu$ of $\mu$ on the support $X = X_{B(\ov W)}$ is finite. Then
\begin{equation}\label{measure_extension_finiteness1}
\sum_{n=1}^\infty \sum_{v\in W_{n+1}}\sum_{w\in W'_n} q^{(n)}_{v,w} \mu(X_v^{(n+1)}(\ov W)) < \infty
\end{equation}
where $q^{(n)}_{v,w}$ are the entries of the stochastic matrix $Q_n$ (see (\ref{StochasticMatrix})) and $X_v^{(n+1)}(\ov W)$
is the tower in the subdiagram $B(\ov W)$ corresponding to the vertex $v \in V_{n+1}$.

(2) If
\begin{equation}\label{measure_extension_finiteness2}
\sum_{n=1}^\infty \sum_{v\in W_{n+1}}\sum_{w\in W'_n} q^{(n)}_{v,w} < \infty,
\end{equation}
then any probability measure $\mu$ defined on the path space $Y$ of the subdiagram $B(\ov W)$
 extends to a finite measure $\wh \mu$ on $X$.
\end{proposition}

\begin{proof} (1) Let $X_w^{(n)}(\ov W)$ be the tower in $B(\ov W)$ corresponding to a
vertex $w\in W_n$. Denote by $h^{(n)}_w(\ov W)$
the height of $X_w^{(n)}(\ov W)$ and by $\ov B_n(w)$ its base, then $\mu(X_w^{(n)}(\ov W)) = h^{(n)}_w(\ov W)\mu(\ov B_n(w))$.
Let $\wh h_w^{(n)}$ be the number of all finite paths from $v_0$ to $w$ contained in $B$, i.e.,
 $\wh h_w^{(n)}$ is the height of $X^{(n)}_w$. Set
$$Z_n = \{x \in X_B : r(x_m)\in W_m \mbox{ for }m\geq n \mbox{ and }x_k\in W_k'\mbox{ for some }k<n\}.$$
Then $\wh \mu (X_B)<\infty$ is finite if and only if $\wh\mu(\bigcup_n Z_n)<\infty$.
Observe that
$$\wh\mu \left(\bigcup_n Z_n \right) =
\sum_{n=1}^\infty \sum_{v\in W_{n+1}}\sum_{w\in W'_n} \frac{f^{(n)}_{v,w} \wh h_w^{(n)} }{h^{(n+1)}_v(\ov W)} \mu(X_v^{(n+1)}(\ov W))$$
Since $q^{(n)}_{v,w} = \frac{f^{(n)}_{v,w} \wh h_w^{(n)} }{\wh h^{(n+1)}_v}
\leq \frac{f^{(n)}_{v,w} \wh h_w^{(n)} }{h^{(n+1)}_v(\ov W)}$, the finiteness of $\wh\mu (\bigcup_n Z_n)$ implies (\ref{measure_extension_finiteness1}).

(2) Suppose (\ref{measure_extension_finiteness2}) holds. Denote
$$
I_n = \sum_{w\in W_n} \wh h_w^{(n)} \mu(\ov B_n(w)).
$$
To prove the finiteness of $\wh\mu(X)$, it suffices to show that the sequence
$\{I_n\}$ is bounded since $\lim_n I_n = \wh \mu(X_B)$. We have
\begin{equation}
I_n = \sum_{w\in W_n} \frac{ \wh h_w^{(n)}}{h_w^{(n)}(\ov W)} h_w^{(n)}(\ov W) \mu(\ov B_n(w)) =
 \sum_{w\in W_n} \frac{ \wh h_w^{(n)}}{h_w^{(n)}(\ov W)} \mu(X_w^{(n)}(\ov B))
\end{equation}
Next, if we show that there exists $M$ such that for all $n$ and $w\in W_n$
\begin{equation}\label{boundnessI_n}
\frac{ \wh h_w^{(n)}}{h_w^{(n)}(\ov W)} \leq M,
\end{equation}
then we obtain that
$$
I_n \leq M \sum_{w\in W_n} \mu(X_w^{(n)}(\ov B)) \leq M.
$$
Let
$$
M_n = \max\{{\frac{ \wh h_w^{(n)}}{h_w^{(n)}(\ov W)} : w \in W_n}\}.
$$
Fix a vertex $v \in W_{n+1}$ and consider
\begin{eqnarray*}
\frac{\wh h_v^{(n+1)}}{h_v^{(n+1)}(\ov W)} & =& \frac1{h_v^{(n+1)}(\ov W)}
\left(\sum_{w\in W_n} f_{v,w}^{(n+1)}\wh h_w^{(n)} +
\sum_{w\in W'_n} f_{v,w}^{(n+1)}\wh h_w^{(n)}\right)\\
& \leq & \frac{M_n}{h_v^{(n+1)}(\ov W)}\sum_{w\in W_n} f_{v,w}^{(n+1)} h_w^{(n)}(\ov W) +
\frac1{h_v^{(n+1)}(\ov W)}\sum_{w\in W'_n} f_{v,w}^{(n+1)} \wh h_w^{(n)}\\
& = & M_n + \frac{\wh h_v^{(n+1)}}{h_v^{(n+1)}(\ov W)}
\sum_{w\in W'_n} f_{v,w}^{(n+1)} \frac{\wh h_w^{(n)}}{\wh h_v^{(n+1)}}\\
& = & M_n + \frac{\wh h_v^{(n+1)}}{h_v^{(n+1)}(\ov W)}\sum_{w\in W'_n} q_{v,w}^{(n)}\\
& \leq & M_n + \frac{\wh h_v^{(n+1)}}{h_v^{(n+1)}(\ov W)} \varepsilon_n,
\end{eqnarray*}
where
$$
\varepsilon_n = \sum_{v\in W_{n+1}}\sum_{w\in W'_n} q^{(n)}_{v,w}.
$$
It follows from the above inequalities that
$$
\frac{\wh h_v^{(n+1)}}{h_v^{(n+1)}(\ov W)}(1- \varepsilon_n) \leq M_n
\mbox{ and }
M_{n+1} \leq \frac{M_n}{1- \varepsilon_n}.
$$
Finally,
$$
 M_{n} \leq \frac{M_1}{\prod_{k=1}^\infty(1- \varepsilon_n)}
$$
where the product is convergent in view of (\ref{measure_extension_finiteness2}).
\end{proof}

\begin{corollary} \label{cor-finiteness}
In the setting of Proposition~\ref{finiteness_condition}, if the subdiagram $B(\ov W)$ has exact finite rank, then
(\ref{measure_extension_finiteness2}) is necessary and sufficient for the finiteness of the extension $\wh \mu$.
\end{corollary}

\begin{proof}
This is immediate from Proposition~\ref{finiteness_condition} and the definition of exact finite rank.
\end{proof}

In the remaining part of this section, we consider finite rank Bratteli diagrams $B$ with incidence matrices of the form
\begin{equation}\label{IncidenceMatrices}
F_n = \left(
      \begin{array}{cc}
        D_n & 0\\
        A_n & C_n\\
      \end{array}
    \right), \ \ \ n\geq 1,
\end{equation}
where matrices $D_n$ and $C_n$ are primitive and $A_n$ is non-zero for all $n$. Then
 the subdiagrams $B(D)$ and $B(C)$, with the incidence matrices
$D_n$ and $C_n$, are simple. By construction, the minimal component of $B$
corresponds to $B(D)$ and the non-minimal one is determined by $B(C)$.
Suppose $\mu$ is a probability invariant measure on $B(C)$. Denote by $\widehat\mu$ the
extension of $\mu$ to $X_B$. Let $A_i =(a^{(i)}_{v,u})$ and set
$$
\alpha_i = \max\{a^{(i)}_{v,u} : v\in V_{i+1}(C),\ u \in V_i(D)\},
$$
$$
\beta_i = \min\{a^{(i)}_{v,u} : v\in V_{i+1}(C),\ u \in V_i(D)\}.
$$

Using Propositions \ref{PropositionExactHeightGrowthEquivMeasureGrowth}
and \ref{finiteness_condition}, we can establish the following result.

\begin{theorem}\label{measure extension} Let the Bratteli diagram $B$ be as above. Suppose that the Bratteli subdiagrams
$B(C)$ and $B(D)$ are of exact finite rank. Assume further that in each of the subdiagrams the heights of towers have the same
asymptotic growth. More precisely, $\inf_{n,v,w}h^{(n)}_v/h^{(n)}_w >0$
where $v$ and $w$ run over the vertices of $B(D)$, and also over the vertices of $B(C)$, and $h^{(n)}_v, h^{(n)}_w$ denote the
heights of the towers {\em within} the corresponding subdiagram.

(i) If
$$
\sum_{i=1}^\infty \alpha_i \frac{||D_{i-1}\cdots D_1||_1}{||C_i\cdots C_1||_1} < \infty,
$$
then the measure $\widehat\mu(X_B)$ is finite.

(ii) If $\widehat\mu(X_B)$ is finite, then
$$
\sum_{i=1}^\infty \beta_i \frac{||D_{i-1}\cdots D_1||_1}{||C_i\cdots C_1||_1} < \infty.
$$
\end{theorem}

\begin{proof} To prove the theorem it is enough to check
  the convergence of the series
\begin{equation}\label{EquationAlgebraicExtension}\sum_{i=1}^\infty
\sum_{v\in W_{i+1}}\sum_{w\in W'_i} q^{(i)}_{v,w},\end{equation}
where $W_i = V(C)\cap V_i$ and $W_i' = V(D) \cap V_i$.

 We observe that it follows from the form of the diagram $B$ that the heights
 $h_w^{(i)}$, for $w\in V(D)$, are completely determined by the
 products of the matrices $D_{i-1}\cdots D_1$. In view of Proposition \ref{PropositionExactHeightGrowthEquivMeasureGrowth} we see that there are
positive constants $k_1$ and $k_2$ such that
$$
k_1\leq\frac{h_w^{(i)}}{||D_{i-1}\cdots D_1||_1}\leq k_2
$$
for all levels $i\geq 1$ and all $w\in V(D)$. (Although Proposition \ref{PropositionExactHeightGrowthEquivMeasureGrowth}(ii) says
``after appropriate telescoping'', we only need the weaker property that there are two-sided estimates. In that proposition, we have that
$||F_{n-1}\cdots F_1||_1$ is the sum of
heights. Since the ratios between heights are bounded away from zero, $h_v^{(n)} / ||F_{n-1}\cdots F_1||_1$ is bounded
from zero and infinity.)

On the other hand, the
finiteness of the extension $\widehat \mu$ is equivalent to the fact
that there exist positive constants $r_1$ and $r_2 $ such that for
all $i\geq 1$ and $v\in V(C)$
 $$r_1\leq \frac{h_v^{(i)}}{||C_{i-1}\cdots C_1||_1}\leq r_2.
 $$

Then, for all $i \geq 1$ we have that
\begin{eqnarray*}
\sum_{v\in W_{i+1}}\sum_{w\in W'_i} q^{(i)}_{v,w} & = &
\sum_{v\in W_{i+1}}\sum_{w\in W'_i}
f^{(i)}_{v,w}\frac{h_w^{(i)}}{h_v^{(i+1)}}\\
&\leq & \alpha_i\frac{k_2
|W_i'|\cdot |W_{i+1}|}{r_1} \frac{||D_{i-1}\cdots
D_1||_1}{||C_i\cdots C_1||_1}.
\end{eqnarray*}

Thus, statement (i) implies the convergence of
(\ref{EquationAlgebraicExtension}) and, therefore, establishes the
finiteness of the extension.

The statement (ii) is proved analogously from the lower bound
for the sum $\sum_{v\in W_{i+1}}\sum_{w\in W'_i}
q^{(i)}_{v,w}$.
\end{proof}

\begin{corollary}\label{CorollaryNecessarySufficientConditions} Let $B$ be as in Theorem \ref{measure extension}.
If there are positive integers $N_1$ and $N_2$ such that $N_1\leq \beta_i\leq \alpha_i \leq N_2$ for all $i\geq 1$, then
\begin{equation}\label{NecessarySufficientConditions1}
\mu(X_B) < \infty \Longleftrightarrow \sum_{i=1}^\infty \frac{||D_{i-1}\cdots D_1||_1}{||C_i\cdots C_1||_1} < \infty.
\end{equation}
\end{corollary}

\begin{remark}\label{RemarkLinearlyRecurrent} (1) The condition $N_1\leq \beta_i\leq \alpha_i \leq N_2$ ($i\geq 1$)
 is equivalent to the property of finiteness of the set $\{A_i : i\geq 1\}$ (recall that we consider Bratteli diagrams with incidence matrices (\ref{IncidenceMatrices}). In particular, this is the case when the matrices $F_i$ are taken from a finite set of matrices (linearly recurrent case, which is discussed below).

(2) For any fixed sequences $\{D_i\}$ and $\{C_i\}$, the condition $\mu(X_B) = \infty$ can be
obtained by an appropriate choice of matrices $A_i$.

(3) In the case of stationary diagrams, Corollary
\ref{CorollaryNecessarySufficientConditions} is a generalization of
the fact that the measure extension is finite if and only if the
spectral radius of $C=C_n$ is strictly greater than that of $D=D_n$,
see Theorem 4.3 in \cite{bezuglyi_kwiatkowski_medynets_solomyak:2010}.
\end{remark}

%
%
%

\subsection{Extension from Odometers}
We consider an important special case of Proposition \ref{finiteness_condition}.
Let $B$ be a finite rank Bratteli diagram with incidence matrices $F_n$. Take a sequence $\ov v = (v_0, v_1,...)$
of vertices in $B$ such that $v_i \in V_i$ and denote by $Y_{\ov v}$ the corresponding ``odometer'', i.e.,
 $Y_{\ov v}$ is the set of paths $x = (x_i)$ such that $r(x_i) = v_i$ for all $i$. Let $\mu_{\ov v}$ be
 the ergodic measure on $Y_{\ov v}$ such that
$$
\mu_{\ov v}([e(v_0, v_n])) = \left(\prod_{i=1}^{n-1} f^{(i)}_{v_{i+1},v_i}\right)^{-1}.
$$
Let $\wh \mu_{\ov v}$ be the extension of $\mu_{\ov v}$. Any odometer is trivially of exact finite rank (since it has rank one!), so
it follows from Corollary~\ref{cor-finiteness} that
\begin{equation}\label{Odometer_case}
\wh \mu_{\ov v}(X_B) < \infty \Longleftrightarrow \sum_{i=1}^\infty
(1 - q^{(i)}_{v_{i+1},v_i}) < \infty
\end{equation}
where $q^{(i)}_{v_{i+1},v_i}$ are the entries of the corresponding stochastic
matrix (\ref{StochasticMatrix}) taken along the sequence $\ov v$.

\begin{corollary} Let $\ov v = (v_0, v_1,...)$ and $\ov w = (w_0, w_1,...)$ be two sequences
of vertices of a finite rank diagram $B$ such that the corresponding measures $\wh \mu_{\ov w}$ and $\wh \mu_{\ov v}$ are finite.
Then there exists a level $n_0$ such that for all $n \geq n_0$ either $w_n = v_n$ or $w_n \neq v_n$.
\end{corollary}
\begin{proof}
Indeed, it follows from (\ref{Odometer_case}) that, without loss of generality,
 one can assume that for all $n$ the inequality $q^{(n)}_{v_{n+1},v_n} > 1/2$
  holds. Since the vector $(q^{(n)}_{v,w})_w$ is probability, there
  exists at most one vertex $w\in V_n$ such that, for a given $v\in V_{n+1}$, the entry
$q^{(n)}_{v,w}$ is greater than 1/2.
\end{proof}

Now we consider several examples which illustrate different cases of
the proved theorems. In particular, one of the examples shows that if
a component $Y_\alpha$ of a Bratteli diagram $B$ supports several
ergodic probability measures, then some of them might give rise to
finite measures and some to infinite ones on $\mathcal E(Y_\alpha)$.
We observe that our examples have some similarities with the
examples constructed in \cite{ferenczi_fisher_talet}, but we use a
completely different approach here. In all the examples below we
extend ergodic measures from subdiagrams which have the simplest
form possible, i.e. they have only one vertex at each level. We should
note that not every measure can be obtained as an extension from
such an elementary subdiagram.

\begin{example}\label{ExampleExtensionOfMeasureFromOdometers}\label{2 times 2} Let $B$ be the Bratteli diagram with incidence matrices
$$F_n =\left(
    \begin{array}{cc}
      b_n & 1 \\
      1 & c_n \\
    \end{array}
     \right),\ \ \ n\geq 1.
$$
Then $B$ contains two natural subdiagrams $B_1$ and $B_2$ defined by odometers $\{b_n\}$ and $\{c_n\}$ ``sitting''
on left and right vertices $v_1$ and $v_2$, respectively. Let $\mu_1$ and $\mu_2$ be the two invariant probability
measures on $B_1$ and $B_2$, respectively. Consider the extensions ${\wh \mu}_1$ and ${\wh \mu}_2$ of measures $\mu_1$
and $\mu_2$ on $X_1 = \E(Y_1)$ and $X_2 = \E(Y_2)$. To compute ${\wh \mu}_1(X_1)$, we use the relation (for ${\wh \mu}_2(X_2)$ we have similar formulas)
$$
{\wh \mu}_1(X_1) =\lim_{n\to\infty} {\wh \mu}_1(X_1(n))
$$
where $X_1(n) = \{x = (x_i)\in X_B : r(x_i) = v_1,\ i\geq n\}$. Notice that for $n\geq 1$
$$
h_1^{(n)} = b_{n-1}h_1^{(n-1)} + h_2^{(n-1)},
$$
$$
h_2^{(n)} = c_{n-1}h_2^{(n-1)} + h_1^{(n-1)}.
$$
Then
\begin{eqnarray*}
      {\wh \mu}_1(X_1(n)) &=& {\wh \mu}_1(X_1(1)) + \sum_{i=2}^n({\wh \mu}_1(X_1(i)) - {\wh \mu}_1(X_1(i-1)) \\
        &=& 1 + \sum_{i=2}^n(\frac{h_1^{(i)}}{b_{i-1}\cdots b_1} - \frac{h_1^{(i-1)}}{b_{i-2}\cdots b_1}) \\
         &=& 1 + \sum_{i=2}^n(\frac{b_{i-1}h_1^{(i-1)} + h_2^{(i-1)}}{b_{i-1}\cdots b_1} - \frac{h_1^{(i-1)}}{b_{i-2}\cdots b_1}) \\
         & = & 1 + \sum_{i=2}^n \frac{h_2^{(i-1)}}{b_{i-1}\cdots b_1}
\end{eqnarray*}
Finally,
\begin{equation}\label{series}
{\wh \mu}_1(X_1) = 1 + \sum_{i=1}^\infty \frac{h_2^{(i)}}{b_{i}\cdots b_1}.
\end{equation}
Thus,
$$
{\wh \mu}_1(X_1) < \infty \Longleftrightarrow \sum_{i=1}^\infty \frac{h_2^{(i)}}{b_{i}\cdots b_1} < \infty.
$$

We note that the function $h_2^{(i)}$ depends on $b_1,...,b_{i-2}$ and $c_1,...,c_{i-1}$. Based on this observation,
 we can easily show that the following statement holds:
\medbreak

{\it For any sequence $\{c_n\}$, there exists a sequence $\{b_n\}$ such that ${\wh \mu}_1(X_1) < \infty$. Similarly,
given a sequence $\{b_n\}$, one can find a sequence $\{c_n\}$ such that ${\wh \mu}_2(X_1) < \infty$. Moreover, one
can construct sequences $\{b_n\}$ and $\{c_n\}$ to obtain
both measures ${\wh \mu}_1$ and ${\wh \mu}_2$ simultaneously either finite or infinite.}
\medbreak

Indeed, formula (\ref{series}) says that, independently of $h_2^{(i)}$,
 we can always choose $b_i$
to ensure the convergence of the series $\sum_{i=1}^\infty h_2^{(i)}(b_{i}\cdots b_1)^{-1}$.
This is possible because $b_i$ is not involved in the formula for $h_2^{(i)}$. Clearly,
this kind of argument proves the claim above.
\medbreak

Now we consider the following Bratteli diagram $\ov B$:

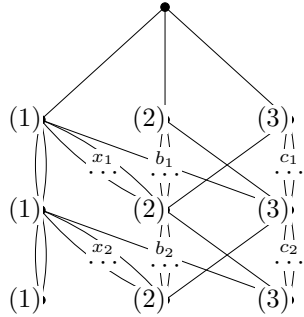
\unitlength = 0.3cm
 \begin{center}
 \begin{graph}(11,15)
\graphnodesize{0.4}
 \roundnode{V0}(5.5,14)
\roundnode{V11}(0.0,9) \roundnode{V12}(5.5,9)
\roundnode{V13}(11.0,9)
\roundnode{V21}(0.0,5) \roundnode{V22}(5.5,5)
\roundnode{V23}(11.0,5)
\roundnode{V31}(0.0,1) \roundnode{V32}(5.5,1)
\roundnode{V33}(11.0,1)
\edge{V11}{V0} \edge{V12}{V0} \edge{V13}{V0}
\bow{V21}{V11}{0.06} \bow{V21}{V11}{-0.06}
 \bow{V22}{V11}{0.06}\bow{V22}{V11}{-0.06}
 \bow{V22}{V12}{0.06} \bow{V22}{V12}{-0.06}
\edge{V22}{V13} \edge{V23}{V11} \edge{V23}{V12} \bow{V23}{V13}{0.06}
\bow{V23}{V13}{-0.06}
%

 \edgetext{V22}{V12}{$\stackrel{b_1}{\cdots}$}
  \edgetext{V22}{V11}{$\stackrel{x_1}{\cdots}$}
  \edgetext{V23}{V13}{$\stackrel{c_1}{\cdots}$}

\bow{V31}{V21}{0.06} \bow{V31}{V21}{-0.06} \bow{V32}{V21}{0.06}
\bow{V32}{V21}{-0.06} \bow{V32}{V22}{0.06} \bow{V32}{V22}{-0.06}
\edge{V32}{V23} \edge{V33}{V21} \edge{V33}{V22} \bow{V33}{V23}{0.06}
\bow{V33}{V23}{-0.06}
%

  \edgetext{V32}{V21}{$\stackrel{x_2}{\cdots}$}
 \edgetext{V32}{V22}{$\stackrel{b_2}{\cdots}$}
  \edgetext{V33}{V23}{$\stackrel{c_2}{\cdots}$}

\nodetext{V11}(-0.7,0){(1)} \nodetext{V21}(-0.7,0){(1)}
\nodetext{V31}(-0.7,0){(1)}

\nodetext{V12}(-0.7,0){(2)} \nodetext{V22}(-0.7,0){(2)}
\nodetext{V32}(-0.7,0){(2)}

\nodetext{V13}(-0.7,0){(3)} \nodetext{V23}(-0.7,0){(3)}
\nodetext{V33}(-0.7,0){(3)}
\end{graph}
 \end{center}

The incidence matrices of $\ov B$ have the form:
$$
F_n = \left(
          \begin{array}{ccc}
            2 & 0 & 0 \\
            x_n & b_n & 1 \\
            1 & 1 & c_n \\
          \end{array}
        \right).
$$

We have proved above that there are sequences $\{b_n\}$ and $\{c_n\}$ such that the subdiagram
 $B$ of $\ov B$ has two finite ergodic measures ${\wh \mu}_1$ and ${\wh \mu}_2$. Let $\ov\mu_1$
 and $\ov\mu_2$ be extensions of
${\wh \mu}_1$ and ${\wh \mu}_2$ from $B$ to $\ov B$. In other words, we extend these measures to
 path spaces $\E(X_i),\ i=1,2$, in the diagram $\ov B$. Direct computations, similar to those above, show that one can choose
sequences $\{x_n\}$, $\{b_n\}$, and $\{c_n\}$ such that the measure $\ov\mu_1$ is
 infinite and the measure $\ov\mu_2$ is finite.

\end{example}

\begin{remark} (1) One can slightly modify Example \ref{2 times 2} and consider the sequence of incidence matrices
$$F_n =\left(
    \begin{array}{cc}
      b_n & s_n \\
      t_n & c_n \\
    \end{array}
     \right),\ \ \ n\geq 1
$$
such that the additional condition $b_n + s_n = t_n + c_n = h_n$ holds. Then the corresponding stochastic matrix $Q_n$ has the form
$$Q_n = \left(
    \begin{array}{cc}
      \frac{b_n}{h_n} & 1 - \frac{b_n}{h_n}\\
      1 - \frac{c_n}{h_n} & \frac{c_n}{h_n} \\
    \end{array}
     \right) =
\left(
    \begin{array}{cc}
      1- \varepsilon_n & \varepsilon_n\\
      \eta_n & 1 - \eta_n \\
    \end{array}
     \right)
$$
because $h_v^{(n+1)} = h_n h_v^{(n)}$ for any vertex $v$. It is not hard to show
that if $\sum_n (\varepsilon_n + \eta_n) < \infty$, then there are two finite
ergodic invariant measures and if $\sum_n (\varepsilon_n + \eta_n) =\infty$,
then the diagram constructed by $\{F_n\}$ is uniquely ergodic.

(2) We also note that the method of Example \ref{2 times 2} can be
applied to construct a simple diagram with $d$ vertices at
each level, having exactly $k$ finite ergodic measures, $k\leq d$.
\end{remark}

%
%
\subsection{Linearly Recurrent Diagrams}

\begin{definition} A Bratteli diagram is called {\it linearly
recurrent} if it has a finitely many different incidence matrices.
\end{definition}
Minimal linearly recurrent diagrams were
studied in the papers \cite{cortez_durand_host_maass:2003} and
\cite{durand_host_scau:1999}. These diagrams appeared there as
Bratteli-Vershik models for minimal dynamical system whose time of
recurrence behaves as a linear function. We should emphasize that
for the needs of our paper the term ``linearly recurrent'' just
means that the set of matrices is finite and we are not interested
here in the time of recurrence.

We begin with the following illustrative example.

\begin{example} Let the diagram $B$ be defined by the incidence matrices
$$F_n =\left(
    \begin{array}{cc}
      \tau_n & 0 \\
      a_n & \omega_n \\
    \end{array}
     \right),\ \ \ n\geq 1,
$$
where the entries of $F_n$ are positive integers (greater than one). Let $\mu$ be the probability measure defined by
the odometer $\{\omega_i\}$. It can be easily shown that
\begin{equation}\label{MeasureofX_B}
\widehat\mu(X_B) = 1 + \sum_{i =1}^\infty a_i\frac{\tau_{i-1}\cdots \tau_1}{\omega_i\cdots \omega_1}.
\end{equation}
(we skipped a routine computation). Then for a particular case when $\omega_n \in \{2,3\}$, $w_1 = 3$, $a_n=1$ and $\tau_n = 2$, we obtain
\begin{equation}\label{FinalFormula}
\widehat\mu(X_B) = 1 + \sum_{i=1}^\infty
\frac{2^{i-1}}{\omega_i\cdots \omega_1} = 1 + \frac1{2} \sum_{n=1}^\infty \left(\frac{2}{3}\right)^{n} (i_{n+1} - i_n),
\end{equation}
 where $ 1 = i_1<i_2<\ldots<i_n<\ldots,$ are all the numbers with $w_{i_n} = 3$.
Relation (\ref{FinalFormula}) yields
a number of sufficient conditions for finiteness of $\widehat\mu(X_B)$. In particular, suppose that
$$
i_{n+1} - i_n \leq Kn^c,\ \ K,c \in \mathbb R_+
$$
for sufficiently large $n$. Then $\widehat\mu(X_B) < \infty$.
\end{example}

%
%
%

 Now we will extend this example to the
case of linearly recurrent diagrams. Let $B = (V,E)$ be a linearly recurrent Bratteli diagram with
incidence matrices $\{F_n\}_{n\geq 1}$. Denote by $\mathcal A$ the
set of all different incidence matrices. Then the diagram $B$
naturally defines a sequence $\omega\in \mathcal A^\mathbb N$ with
$\omega_i = F_i$. It turns out that the growth rate of the product
$||F_n\cdots F_1||_1$ heavily depends on the combinatorial
properties of the sequence $\omega$. The next proposition, which was
essentially proved in \cite{johnson_bru:1990}, is a crucial step for
getting estimates for the growth of matrix products.

Let $R$ be a diagonal matrix with positive diagonal entries. Set
$$M(R) = \max_{i,j} R_{i,i} R^{-1}_{j,j},\ \ \ \ m(R) =
\min_{i,j}R_{i,i}R^{-1}_{j,j}.
$$
Then for any non-negative matrix $A$, we have the inequalities
$$ m(R)||A ||_1\leq ||R^{-1} A R ||_1 \leq M(R) ||A ||_1.
$$

For a positive vector $x$, denote by $D_x$ the diagonal $d\times d$
matrix whose diagonal entries are the entries of $x$ written in
the same order. For two positive vectors $x$ and $y$, denote by $x/y$ their
componentwise ratio, i.e., $x/y = (x_1/y_1,...,x_d/y_d)$.
For a vector $x >0 $, let $x_{\max}$ be the maximal entry of $x$
and $x_{\min}$ the minimal one.

\begin{proposition}\label{propositionMatrixGrowth} Let $A_1,\ldots,A_n$ be primitive matrices.
Let $x_i$ denote a Perron-Frobenius eigenvector for the matrix
$A_i$ and $\rho(A_i)$ its spectral radius. Then

$$ \frac{||A_1A_2\cdots A_n||_1}{\rho(A_1)\rho(A_2)\cdots \rho(A_n)}\leq \frac{1}{m(D_{x_n})}
\left(\frac{x_n}{x_{n-1}} \right)_{\max} \cdots
\left(\frac{x_2}{x_1}\right)_{\max}\left(\frac{x_1}{x_{n}}
\right)_{\max}
$$
and
$$
\frac{||A_1A_2\cdots A_n||_1}{\rho(A_1)\rho(A_2)\cdots \rho(A_n)}\geq \frac{1}{M(D_{x_n})}
\left(\frac{x_n}{x_{n-1}} \right)_{\min} \cdots
\left(\frac{x_2}{x_1}\right)_{\min}\left(\frac{x_1}{x_{n}}
\right)_{\min}
 $$
\end{proposition}
\begin{proof} It was shown in the proof of Theorem 1 from
\cite{johnson_bru:1990} that
$$
\begin{array}{ll} D_{x_n}^{-1} A_1A_2\ldots A_n D_{x_n} \overline 1 \\
\\
\leq \rho(A_1)\rho(A_2)\ldots \rho(A_n) \left(\frac{x_n}{x_{n-1}}
\right)_{\max} \cdots
\left(\frac{x_2}{x_1}\right)_{\max}\left(\frac{x_1}{x_{n}}
\right)_{\max} \overline 1
\end{array}
$$
where $\overline 1 = (1,\ldots,1)^T$.
We note that
$$
|| A_1A_2\cdots A_n ||_1 \leq
\frac{1}{m(D_{x_n})}||D_{x_n}^{-1} A_1A_2\cdots A_n D_{x_n} ||_1
$$
and then we use (\ref{eq-norm2}) to prove the first inequality. The second
one follows from the proof of \cite[Theorem 1]{johnson_bru:1990} in a similar way by reversing the inequalities.
\end{proof}

\medbreak
Next, consider the sequence $\omega \in\mathcal A^\mathbb
N$ defined by a linearly recurrent Bratteli diagram $B$ as above.
Let $I_A(n)$ be the number of occurrences of the letter $A$ in the word
$\omega_1\omega_1\ldots \omega_{n}$. Let $\mathcal A^{(2)}$ be
the set of all words of length two from the sequence $\omega$. Denote by $I_{AB}(n)$ the number of occurrences
 of the pair $AB$ in the word $(\omega_1\omega_2)(\omega_2\omega_3)\ldots (\omega_n\omega_{n+1})$.

\begin{definition} We will say that the linearly recurrent diagram
$B$ is {\it regular} if for every matrix $A\in \mathcal A$ and every
 pair $AB\in \mathcal A^{(2)}$ the limits
$$
 d(A) = \lim_{n\to \infty} \frac{I_A(n)}{n},\ \ \ d(AB) = \lim_{n\to \infty} \frac{I_{AB}(n)}{n}
$$
exist. We call $d(A)$ the density of $A$ in $\omega$ and $d(AB)$ the
density of $AB$ in the sequence $(\omega_1\omega_2)(\omega_2
 \omega_3)(\omega_3\omega_4)\ldots$
\end{definition}

Let $x_A$ be a Perron-Frobenius eigenvector of
$A\in\mathcal A$. For any pair of
 matrices $A$ and $B$ with $AB\in \mathcal A^{(2)}$, denote by
 $\overline r(A,B) $ the ratio $(x_B/x_A)_{\max}$.
 Similarly, we set $\underline r(A,B) $ to be the ratio $(x_B/x_A)_{\min}$.
 Finally, we set
$$\overline \rho(\omega) = \prod_{A \in \mathcal A}
\rho(A)^{ d(A)} \times \prod_{AB \in \mathcal A^{(2)}} \overline
r(A,B)^{ d(AB)}.
$$

We refer to the number $\overline \rho(\omega)$ as the {\it
upper spectral radius along the sequence $\omega$}. The number
$\underline \rho(\omega)$ is defined similarly by using the values
$\underline{r}(A,B)$.

The next lemma shows that $\overline \rho(\omega)$ and $\underline
\rho(\omega)$ are well-defined and may serve as the upper and lower
bounds for the products of incidence matrices, respectively.

\begin{lemma}\label{lemmaJointSpectralRadius} Let $B$ be a regular linearly recurrent diagram with the sequence
of primitive incidence matrices $\omega\in\mathcal A^\mathbb N$.
Then

(1) $\overline \rho(\omega)$ and $\underline \rho(\omega)$ do not
depend on the choice of eigenvectors $x_A$, $A\in\mathcal A$;

(2) the following inequalities hold

$$\liminf_{n\to\infty}\left(||\omega_1\omega_2\ldots\omega_{n}||_1\right)^{\frac{1}{n}}\geq \underline\rho(\omega)$$
and
$$\limsup_{n\to\infty}(||\omega_1\omega_2\ldots\omega_{n}||_1)^{\frac{1}{n}}\leq \overline\rho(\omega).$$
\end{lemma}
\begin{proof} (1) Let $x_A$ be a Perron-Frobenius eigenvector of $A$ and $x_A' = c_A x_A$, $c_A >0$. For each $n$, define
$$\rho_n =
\prod_{i=1}^{n}\rho(\omega_{i})\cdot \overline
r(\omega_i,\omega_{i+1}).$$ Let the number $\rho_n'$ be defined
similarly to $\rho_n$, but with the eigenvectors $x_A$ and $x_B$
replaced by $x'_A$ and $x_B'$.
 Then, it is not hard to check that
 $$\rho_n' =
\frac{c_{w_{n+1}}}{c_{w_1}}\rho_n\mbox{ for all }n.
$$
Since the set $\{c_A : A\in\mathcal A\}$ is finite, we get that
$$\lim\limits_{n\to\infty} \left(\frac{\rho_n}{\rho_n'}\right)^{\frac{1}{n}} = 1.$$
On the other hand, we see that $$(\rho_n)^{\frac{1}{n}} =
\prod_{A\in\mathcal A}\rho(A)^{\frac{I_A(n)}{n}} \times
\prod_{AB\in\mathcal A^{(2)}} \overline
r(A,B)^{\frac{I_{AB}(n)}{n}}\to \overline \rho(\omega)
$$
as $n\to\infty$.
 This shows that the definition of $\overline \rho(\omega)$ does not
depend on the choice of Perron-Frobenius eigenvectors.
The proof for $\underline
\rho(\omega)$ is similar and left to the reader.

(2) Using Proposition \ref{propositionMatrixGrowth} and the fact
that the set of matrices is finite, we can find a constant $K>0$,
which does not depend on $n$, such that
\begin{eqnarray*}(||\omega_1\omega_2\ldots \omega_n||_1)^{\frac{1}{n}} &
\leq & \left(K\prod_{i=1}^{n}\rho(\omega_{i})\cdot \overline
r(\omega_i,\omega_{i+1})\right)^\frac{1}{n} \\
& = & K^\frac{1}{n}\prod_{A\in\mathcal A}\rho(A)^{\frac{I_A(n)}{n}}
\times \prod_{AB\in\mathcal A^{(2)}}r(A,B)^{\frac{I_{AB}(n)}{n}}
\\
& \to & \overline \rho(\omega)
\end{eqnarray*}
as $n\to\infty$.
Thus, $\overline \rho(\omega)\geq
\limsup_{n\to\infty}(||\omega_1\omega_2\ldots\omega_{n}||_1)^{\frac{1}{n}}$.
The other inequality is established in a similar way.
\end{proof}

Let $B$ be a regular linearly recurrent Bratteli diagram whose
incidence matrices have the form

$$F_n = \left(\begin{array}{cc} D_n & 0\\
A_n & C_n \end{array}\right),$$ with $D_n$ and $C_n$ being primitive
matrices.

By definition of $B$, the sequences $\{D_n\}_{n\geq 1}$ and
$\{C_n\}_{n\geq 1}$ have only finitely many different matrices.

The following theorem shows that the spectral radii along the sequences $\{D_n\}_{n\geq 1}$ and $\{C_n\}_{n\geq 1}$ can distinguish the
growth rates of the minimal and non-minimal components of $B$. This, in
particular, answers the question of finiteness of the measure
extension from the subdiagram $B(C)$ and allows one to distinguish certain non-orbit equivalent systems.

\begin{theorem}\label{TheoremLinearlyRecurrentCase} Let $B$ be a regular linearly recurrent diagram as above.

(i) If $\underline \rho (\{D_n\}_{n\geq 1}) > \overline \rho
(\{C_n\}_{n\geq 1})$, then the extension of the measure from $B(C)$
is infinite.

(ii) $\overline \rho (\{D_n\}_{n\geq 1}) < \underline \rho
(\{C_n\}_{n\geq 1})$, then the extension of the measure from $B(C)$
is finite.
\end{theorem}

\begin{proof} We note that Proposition \ref{PropositionSufficientCondBoundedAway} implies that the measure
 of towers is bounded away from zero and the tower heights grow with the same speed within the
subdiagram $B(C)$ and $B(D)$.
In view of Corollary
\ref{CorollaryNecessarySufficientConditions} and Remark
\ref{RemarkLinearlyRecurrent}, it is sufficient to verify whether the
series $$\sum_{n=1}^\infty \frac{||D_{n-1}\cdots
D_1||_1}{||C_n\cdots C_1||_1}$$ is convergent or not. Fix $\e>0$ so that $ \underline \rho(\{D_n\}_{n\geq 1}) -
\e>\overline \rho(\{C_n\}_{n\geq 1})+\e$. Set $$r =
\limsup_{n\to\infty} \left(\frac{||D_{n-1}\cdots
D_1||_1}{||C_{n}\cdots C_1||_1}\right)^{\frac{1}{n}}.$$ Then, by
Lemma \ref{lemmaJointSpectralRadius}, we get that
\begin{eqnarray*}
\sup_{n\geq k}\left(\frac{||D_{n-1}\cdots D_1||_1}{||C_{n}\cdots
C_1||_1}\right)^{\frac{1}{n}} &\geq &
 \frac{\inf_{n\geq
k}(||D_{n-1}\cdots D_1||_1)^{1/n}}{\sup_{n\geq k}(||C_{n}\cdots
C_1||_1)^{1/n}}\\
& \geq & \frac{\underline \rho(\{D_n\}_{n\geq 1}) -
\e}{\overline \rho(\{C_n\}_{n\geq 1})+\e}\\
& > &1
\end{eqnarray*}
 for all $k$ large
enough. This implies that $r> 1$ and, hence, the series diverges by
the root test. The fact that condition (ii) leads to the convergent series
(with $r<1$) is proved similarly.
\end{proof}

\begin{remark} (1) We observe that the statement (i) in Theorem
\ref{TheoremLinearlyRecurrentCase} implies that the diagram $B$ has
a unique invariant measure supported by the minimal component only.
On the other hand, the statement (ii)
guarantees the existence of a fully supported
invariant measure (along with the measure on the
minimal component).

(2) We also note that it is possible to treat the numbers
$$\overline \lambda(\omega) = \limsup_{n\to\infty} (||\omega_1\ldots
\omega_n||_1)^{\frac{1}{n}}\mbox{ and }\underline \lambda(\omega) =
\liminf_{n\to\infty} (||\omega_1\ldots \omega_n||_1)^{\frac{1}{n}}$$
as the growth rate for matrix products. Then Theorem
\ref{TheoremLinearlyRecurrentCase} still holds if we replace
$\overline\rho(\omega)$ with $\overline\lambda(\omega)$ and
$\underline\rho(\omega)$ with $\underline\lambda(\omega)$.

\end{remark}


\section{Absence of Strong Mixing}\label{SectionAbsenceStrongMixing}

In this section we study mixing properties of Vershik maps on finite rank
Bratteli
diagrams. We will prove that if an invariant measure has the property that
the measure values of all towers are bounded away from
zero (i.e.\ it has exact finite rank), then any Vershik map on such a diagram is not strongly
mixing. This was earlier  proved by A.~Rosenthal \cite{rosenthal:1984} in the context of measure-preserving transformations of exact finite rank
by very different methods, in a hard to find unpublished manuscript.
We then establish the absence of mixing if a Bratteli diagram (not necessarily simple or uniquely ergodic) is
equipped with the so-called {\em consecutive ordering}.

The absence of strong mixing has been earlier established
for substitution systems \cite{dekking_keane:1978}, \cite{bezuglyi_kwiatkowski_medynets_solomyak:2010}, interval exchange transformations
\cite{katok:1980}, and linearly recurrent systems \cite{cortez_durand_host_maass:2003}.
We also mention the Ph.D
thesis of Wargan \cite{wargan:thesis:2002} devoted to the study of
some generalizations of linearly recurrent systems where he proved
the absence of strong mixing for such systems.
Our methods have some common features with those of \cite{katok:1980}.

We start with a preliminary lemma.

\begin{lemma}\label{LemmaAbsenceMixing}
Let $\{(Y_n,\nu_n, S_n)\}_{n\geq 1}$ be a family of probability
measure-preserving transformations, where $Y_n$ is a shift-invariant subset of
$A^\mathbb Z, |A| < \infty,$ and $S_n$ denotes the left shift. Then there is a word $\omega = \omega_0\ldots \omega_{r-1}$ from $A^+$ such that
$\omega_0 = \omega_{r-1}$ and $\limsup_n \nu_n([\omega])>0$.
\end{lemma}

\begin{proof} Set $d = |A|$. Then for every $n$ we have
$\sum_{w\in A^{d+1}}\nu_n([w]) = 1.$
Therefore,
$$
\sum_{w\in A^{d+1}}\limsup_n\nu_n([w])\geq \limsup_n \sum_{w\in A^{d+1}}\nu_n([w]) = 1.
$$
Choose $w\in A^{d+1}$ with $\limsup_n\nu_n([w])>0$. Then the word $w$
contains a subword $\omega$ starting and ending with the same letter, hence $\nu_n([\omega])\ge \nu_n([w])$ for all $n$, and we are done.
\end{proof}

\begin{theorem}\label{TheoremAbsenceStrongMixing} Let $B = (V,E,\leq)$
be an ordered simple Bratteli diagram of exact finite
rank. Let $T:X_B\rightarrow X_B$ be the Vershik map defined by the
order $\leq$ on $B$ ($T$ is not necessarily continuous everywhere).
 Then the dynamical system $(X_B,\mu,T)$ is not strongly mixing
with respect to the unique invariant measure $\mu$.
\end{theorem}

\begin{proof} (I) In the proof we will consider the family
$\{X_v^{(n)} : v \in V_n\}$ as a Kakutani-Rokhlin partition of $X_B$. Then
$X_v^{(n)} = \{B_n(v),\ldots,T^{h_v^{(n)}-1}B_n(v)\}$ is a $T$-tower,
where $h_v^{(n)}$ is the number of finite paths from the top vertex $v_0$
to a vertex $v$ of level $n$, and $B_n(v)$ is the cylinder
set generated by the finite minimal path connecting the vertices $v_0$
and $v$.

Set $B_n = \bigsqcup_v B_n(v)$. Consider the induced system
$(B_n,\mu_n,T_n)$, with $\mu_n= \mu|_{B_n}/\mu(B_n)$, the probability measure invariant with respect to the induced transformation $T_n$.
Set $A = \{1,\ldots,d\}$. Define the map $\pi_n: B_n \rightarrow A^\mathbb Z$ by $\pi_n(x)_i = v$ if and only if $T_n^i(x)\in B_n(v)$. Denote by
$(Y_n,\nu_n,S_n)$ the factor-system determined by $\pi_n$ (i.e.\ $Y_n=\pi_n(B_n)$, and $\nu_n=\pi_n^*\mu_n$).

Applying Lemma \ref{LemmaAbsenceMixing} to the family $\{(Y_n,\nu_n,S_n)\}_{n\geq 1}$, choose a word
$\omega = \omega_0\ldots \omega_{r-1}$ with $\omega_0 = \omega_{r-1}$
and $\limsup_{n}\nu_n([\omega])>0$.

(II) For each infinite path $x\in X_B$, denote by $v_n(x)$ the vertex of level $n$ the path $x$ goes through.
 Fix a level $m$ and apply the pointwise ergodic theorem to the induced system $(B_m,\mu_m,T_m)$ and
 the set $F_m = \pi_m^{-1}([\omega])$. Then for $\mu_m$-a.e. $x\in B_m$, we have \begin{equation}\label{eqErgTheoremAbsenceMixing}
\mu_m(F_m) = \lim\limits_{n\to\infty}\frac{1}{i_n^{(m)}(x)+j_n^{(m)}(x)}\sum\limits_{i = -i_n^{(m)}(x)}^{j_n^{(m)}(x)} 1_{F_m}(T_m^i(x)),
\end{equation}
where $i_n^{(m)}(x)$ is the least integer such that $T_m^{-i_n^{(m)}(x)}$ maps the initial segment of $x$ to the minimal
 finite path from the set $E(v_0,v_n(x))$. Similarly, $j_n^{(m)}(x)$ is the least
integer such that $T_m^{j_n^{(m)}(x)}$ maps the initial segment of $x$ to the maximal
path from $E(v_0,v_n(x))$ within $B_m$. Notice that $j_n^{(m)}(x) +
i_n^{(m)}(x)$ is the number of finite paths from the vertices of level $m$ to the vertex $v_n(x)$.

Define the map $\sigma_n$ from $V_n$ into the set of finite words over $V_{n-1}$ by setting $\sigma_n(v) = s_0\ldots s_{p}$ where
$s_i\in V_{n-1}$ and $\{s_0,\ldots ,s_{p}\}$ are the sources of the edges terminating at $v$ and taken in the order of $\leq$. In other words,
we symbolically encode the order $\leq$. For $n>m$, set $\sigma^{(m,n)} = \sigma_m\circ\cdots\circ\sigma_{n+1}$. Notice that
$|\sigma^{(m,n)}(v)|$ is the number of paths from $v$ to the vertices of level $m$.

The definition of the set $F_m$ implies that $T_m^i(x)\in F_m$ for some $-i_n^{(m)}(x)\leq i\leq j_n^{(m)}(x)$ if and only if the word $\omega$
occurs in $\sigma^{(m,n)}(v_n(x))$ at the position $i + i_n^{(m)}(x)$. Thus, the frequency of $\omega$ in $\sigma^{(m,n)}(v_n(x))$
 is equal to
$$\fr(\omega,\sigma^{(m,n)}(v_n(x))) = \frac{1}{i_n^{(m)}(x)+j_n^{(m)}(x)}\sum\limits_{i = -i_n^{(m)}(x)}^{j_n^{(m)}(x)} 1_{F_m}(T_m^i(x)).$$

Since $\mu(X_v^{(n)})\geq \delta > 0$ for all $n$ and $v$ ($\delta$ is taken from the definition of exact finite rank),
given a vertex $v$, the set of all paths visiting $v$ infinitely many times has measure one
(this follows by ergodicity, as in the proof of Theorem \ref{TheoremGeneralStructureOfMeasures}(II)). Hence
(\ref{eqErgTheoremAbsenceMixing}) implies
\begin{equation}\label{eqAbsenceMixingFreq}\mu_m(F_m) = \lim\limits_{n_k\to\infty}\fr(\omega,\sigma^{(m,n_k)}(v)),\end{equation}
for a subsequence $n_k$, for every vertex $v$.

(III) Since $\limsup_m \mu_m(F_m) = \limsup_m \nu_m([\omega])>0$, the equation (\ref{eqAbsenceMixingFreq}) guarantees
that there is a telescoping of the diagram such that $$\fr(\omega,\sigma^{(m,m+1)}(v))\geq \rho >0$$ for all $m$ and $v$.
For every level $n$,
define the set $S_n$ of all infinite paths $x\in B_n(\omega_0)$ such that the sources
of the first $r-1$ (with respect to $\leq$) successors of the edge $x_{n+1}$ (between levels $n$ and $n+1$) are exactly the vertices $\omega_1,\ldots, \omega_{r-1}$. Set also
$C_n = \bigsqcup_{i= 0 }^{h_{\omega_0}^{(n)}-1} T^i S_n$ ($C_n$ is a subtower of $X_{\omega_0}^{(n)}$). Denote by $(f^{(n)}_{v,w})$
the entries of the $n$-th incidence matrix.
Then
\begin{eqnarray*}\mu(C_n) & = & h_{\omega_0}^{(n)}\sum_{v\in V_{n+1}}\mu(B_{n+1}(v))\fr(\omega_0,\sigma^{(n,n+1)}(v))f^{(n)}_{v,\omega_0} \\
 & \geq & \rho h_{\omega_0}^{(n)}\sum_{v\in V_{n+1}}\mu(B_{n+1}(v))f^{(n)}_{v,\omega_0} \\
& = & \rho \mu(X_{\omega_0}^{(n)}).
\end{eqnarray*}
It follows that there exists $\gamma >0$ such that $\mu(C_n) \geq \gamma >0$ for all $n$.
Set $$q_n = h_{\omega_0}^{(n)}+\ldots+h_{\omega_{r-2}}^{(n)}.$$
Since $\omega_0 = \omega_{r-1}$, we obtain that for all $n\geq 1$ and
$\ell=0,\ldots,h_{\omega_0}^{(n)}-1$,
\begin{equation}\label{EquationMixingReturn}
T^{q_n +\ell}S_n\subset
T^{\ell} B_n(\omega_0).
\end{equation}

(IV) Choose a level $n_0$ such that $\mu(B_{n_0}(v))<\gamma/2$ for all
$v=1,\ldots,d$. For each level $n\geq n_0$, there is a vertex $v_n$
such that $B_n(\omega_0)\subset B_{n_0}(v_n)$. By telescoping we may assume
that $v_n = v$ for all $n$. Set $D_n = C_n\cap B_{n_0}(v)$. We note
that $\mu(D_n) > 0$.

Since the Kakutani-Rokhlin partitions $\{X_v^{(n)}\}$ associated to a Bratteli diagram are nested, we obtain that the sets
$T^\ell B_n(\omega_0)$, for $0\leq \ell < h_{\omega_0}^{(n)}$, either
lie in the set $B_{n_0}(v)$ or are disjoint from it. Hence, by the
definition of $D_n$, we obtain that if $x\in D_n$, then $x\in T^\ell S_n\subset T^\ell B_n(\omega_0)\subset B_{n_0}(v)$, for some $0\leq \ell < h_{u_0}^{(n)}$.
Condition \ref{EquationMixingReturn} implies that $T^{q_n}x \in T^\ell B_n(\omega_0)\subset B_{n_0}(v)$. Thus,
$$T^{q_n}D_n\subset B_{n_0}(v)\mbox{ for all }n\geq 1.
$$
Hence, $D_n\subset B_{n_0}(v)\cap T^{-q_n}B_{n_0}(v)$.
As the
Vershik map is aperiodic, we conclude that $q_n\to\infty$ as
$n\to\infty$.
Thus, the theorem would be proved if we show that
\be \label{eq-claim}\limsup_{n\to \infty} \mu(D_n)/\mu(B_{n_0}(v))\ge \gamma,
\ee
because then for some $n=n_k\to\infty$ we will have
$$\mu(B_{n_0}(v)\cap T^{-q_n}B_{n_0}(v))\geq\mu(D_n)\geq
(\gamma/2) \mu(B_{n_0}(v))>\mu(B_{n_0}(v))^2.$$

(V) By the pointwise ergodic theorem we may find a path $x$ that visits the vertex $\omega_0$ infinitely many times and
such that
$$\frac{|\{-i_n^{(1)}(x)\leq \ell \leq j_n^{(1)}(x) : T^\ell(x)\in B_{n_0}(v)\}|}{h_{v_n(x)}^{(n)}}\to \mu(B_{n_0}(v))\mbox{ as
}n\to\infty.$$ Let $\Nk:=\{n:\ v_n(x) = \omega_0\}$, which is infinite by assumption.
(Here we use the same notation as in (\ref{eqErgTheoremAbsenceMixing}), which is consistent with (\ref{pointwise_ergodic_theorem}); note that
$T_1=T$.)
Then we have for all $n\in \Nk$:
\begin{eqnarray*}
\mu(D_n) & = & \frac{|\{\ell=0,\ldots,h_{\omega_0}^{(n)}-1 :
T^\ell B_n(\omega_0)\subset B_{n_0}(v)\}|}{h_{\omega_0}^{(n)}}\mu(S_n)h_{\omega_0}^{(n)}
\\
& = & \frac{|\{-i_n^{(1)}(x)\leq \ell \leq j_n^{(1)}(x) : T^\ell(x)\in
B_{n_0}(v)\}|}{h_{\omega_0}^{(n)}}\mu(S_n)h_{\omega_0}^{(n)} \\
&\sim&\mu(B_{n_0}(v))\mu(C_n)\geq \gamma \mu(B_{n_0}(v))\mbox { as
}n\to\infty, \ n\in \Nk.
\end{eqnarray*}
This proves (\ref{eq-claim}), and the theorem follows.
\end{proof}

\medskip The last theorem holds for any order on the Bratteli diagram. In the next result we show that a somewhat regular
 ordering allows us to drop the assumption of exact finite rank.

Following \cite[Chapter 6]{berthe:2010}, by the {\it consecutive ordering} we mean an ordering on a diagram
 such that whenever edges $e,f,g$ have the same range, $e\leq f \leq g$, and $e$ and $g$ have the same source, then $f$ has the same source as
$e$ and $g$. We remark that such an ordering is not preserved under the telescoping. Bratteli diagrams with a special case of the consecutive
ordering were discussed in
\cite{mela:2006}, \cite{Bailey:thesis}, \cite{BaileyPetersen08}, and \cite{berthe:2010}.

\begin{theorem} Let $B = (V,E,\leq)$ be an ordered (not necessarily simple) Bratteli diagram of finite rank, where $ \leq$ is a consecutive ordering. Let $T:X_B\rightarrow X_B$ be a Vershik map defined by the
order $\leq$ on $B$ ($T$ is not necessarily continuous everywhere) and $\mu$ a finite $T$-invariant measure. Assume that if two vertices in
consecutive levels are connected by an edge, then there are at least two such edges.
 Then the dynamical system $(X_B,\mu,T)$ is not strongly mixing.
\end{theorem}

\begin{proof} To prove the result, we will use the same idea as in the proof of Theorem \ref{TheoremAbsenceStrongMixing}. First of all,
using Theorem \ref{TheoremGeneralStructureOfMeasures}, we choose a vertex $\omega_{0}$ such that
\begin{equation}\label{EquationAbsenceMixingBoundedAway}
\mu(X_{\omega_0}^{(n_k)})\geq \delta>0\end{equation} along some sequence $n_k\to\infty$ as
$k\to\infty$.
Set
$$
C_n = \{x\in X_B : r(x_n) = \omega_0\mbox{ and the source of the successor of }x_{n+1}\mbox{ is }\omega_0\}.
$$
Clearly, $C_n$ is a subtower of $X_{\omega_0}^{(n)}$. Denote by $(f_{w,v}^{(n)})$ the entries of the $n$-th incidence matrix.
Then the definition of the ordering
implies that
$$\mu(C_n) = h_{\omega_0}^{(n)} \sum_v (f_{v,\omega_0}^{(n)} - 1) \mu(B_n(v)). $$
Here the summation is taken over
all $v$ with $f_{v,\omega_0}^{(n)} > 0$. Since $(f_{v,\omega_0}^{(n)} - 1)/f_{v,\omega_0}^{(n)}\geq 1/2$ whenever $f_{v,\omega_0}^{(n)}>0$, we get that
$$\mu(C_n) \geq \frac{1}{2}h_{\omega_0}^{(n)} \sum_v f_{v,\omega_0}^{(n)}\mu(B_n(v)) = \frac{1}{2}\mu(X_{\omega}^{(n)})\geq \frac{\delta}{2},
$$
when $n$ runs along an infinite sequence.

Set $S_n = C_n\cap B_n(\omega_0)$. It follows from the definition of $C_n$ that for all $n\geq 1$ and
$\ell=0,\ldots,h_{\omega_0}^{(n)}-1$,
$$
T^{q_n +\ell}S_n\subset
T^{\ell} B_n(\omega_0),
$$ where $q_n = h_{\omega_0}^{(n)}$. To complete the proof, it remains to repeat the arguments from
 the proof (part (IV)) of Theorem \ref{TheoremAbsenceStrongMixing}. We leave this to the reader.
\end{proof}


We do not know if there exist aperiodic Bratteli-Vershik systems of finite rank which are strongly mixing. The well-known
Smorodinsky-Adams staircase transformation \cite{adams:1998} is not mixing, but it is constructed using ``spacers''
which implies that the  Bratteli-Vershik model built on its symbolic realization (see \cite{ferenczi:1996})
has a fixed point, and hence not aperiodic (it is also non-simple).

\section{Conclusion.}
In this paper we performed a detailed analysis of invariant measures on finite rank aperiodic Bratteli diagrams, both simple and non-simple. Here are some
of the key findings:

\begin{itemize}

\item We introduced the notion of a Bratteli diagram of exact finite rank, which parallels the same notion in measurable dynamics.

\item Every ergodic measure (finite or infinite $\sigma$-finite) is an extension of a finite invariant measure from a simple subdiagram of exact finite
rank.

\item Exact finite rank implies unique ergodicity.

\item
Exact finite rank and the identical asymptotic growth of towers imply the identical asymptotic behavior of measures of tower
bases.

\item
Exact finite rank and the identical asymptotic behavior of measures of tower bases imply the identical asymptotic growth of
towers.

\item
The equality of tower heights does not guarantee the unique ergodicity and, as a result, exact finite rank.

\item Exact finite rank does not ensure the same asymptotic growth of tower heights and the identical asymptotic behavior
of measures of tower bases.

\item Exact finite rank implies absence of strong mixing for the Vershik map for any ordering on the diagram.

\item Without the exact finite rank assumption, if the ordering of the Bratteli diagram is consecutive, then the Vershik map is not strongly mixing.
\end{itemize}

\noindent\medbreak{\bf Acknowledgement.} This work was done during our visits to the University of Oregon, University of Washington, University of Warmia and Mazury, and Institute for Low Temperature Physics. We are thankful to these
institutions for the hospitality and support. The third named author
is also grateful to the Erwin Schr\"{o}dinger International
Institute for Mathematical Physics in Vienna for its hospitality and
support of the present work. We would like also to thank Karl Petersen for useful discussions and S\'ebastien Ferenczi for sending us a copy of
\cite{rosenthal:1984}.

%

%
%
%


\begin{thebibliography}{CDHM03}

\bibitem[A98]{adams:1998} T.~Adams. \newblock Smorodinsky's conjecture on rank-one mixing.
\newblock {\em Proc. Amer. Math. Soc.} 126(3): 739--744, 1998.

\bibitem[B06]{Bailey:thesis} S.~Bailey, \emph{Dynamical properties of some non-stationary, non-simple Bratteli-Vershik systems},
Ph.D. thesis, The University of North Carolina at Chapel Hill, 2006, 144 pages.

\bibitem[BP08]{BaileyPetersen08} S.~Bailey Frick and
 K.~Petersen. \newblock{Random permutations and unique
 fully supported ergodicity for the Euler adic transformation}. \newblock{\em Ann. Inst. Henri Poincare' Probab. Stat.}
 44(5):876--885, 2008.


\bibitem[BKM09]{bezuglyi_kwiatkowski_medynets:2009}
S.~Bezuglyi, J.~Kwiatkowski, and K.~Medynets.
\newblock Aperiodic substitution systems and their Bratteli diagrams.
\newblock {\em Ergod. Th. \& Dynam. Sys.}, 29(1): 37--72, 2009.

\bibitem[BKMS10]{bezuglyi_kwiatkowski_medynets_solomyak:2010}
S.~Bezuglyi, J.~Kwiatkowski, K.~Medynets, and B.~Solomyak.
\newblock Invariant measures on stationary Bratteli diagrams.
\newblock {\em Ergod. Th. \& Dynam. Sys.}, 30(4): 973--1007, 2010.

\bibitem[Bir57]{Birkhoff:1957}
G.~Birkhoff.
\newblock Extensions of Jentzschi's theorem.
\newblock {\em Trans. Amer. Math. Soc.}, 85: 219--297, 1957.

\bibitem[Bir67]{Birkhoff:1967}
G.~Birkhoff.
\newblock {\em Lattice Theory}, volume~25 of {\em AMS Colloquium Publications}.
\newblock American Mathematical Society, third edition, 1967.



\bibitem[Bos92]{boshernitzan:1992}
M.~Boshernitzan.
\newblock A condition of unique ergodicity of minimal symbolic flows.
\newblock {\em Ergod. Th. \& Dynam. Sys.}, 12: 425--428, 1992.

\bibitem[Bos93]{boshernitzan:1993} M.~Boshernitzan.  \newblock
Quantitative recurrence results.
\newblock {\em Invent. Math.} 113(3):617–631, 1993.

\bibitem[BDM10]{bressaud_durand_maass:2009}
X.~Bressaud, F.~Durand, and A.~Maass.
\newblock On the eigenvalues of finite rank Bratteli-Vershik dynamical systems.
\newblock {\em Ergod. Th. and Dynam. Sys.}, 30(3): 639--664, 2010.

\bibitem[CDHM03]{cortez_durand_host_maass:2003}
M.I. Cortez, F.~Durand, B.~Host, and A.~Maass.
\newblock Continuous and measurable eigenfunctions of linearly recurrent
  dynamical Cantor systems.
\newblock {\em J. London. Math. Soc.}, 67(2): 790--804, 2003.

\bibitem[DL06]{damanik_lenz:2006}
D. Damanik and D. Lenz.
\newblock A condition of Boshernitzan and uniform convergence in the multiplicative ergodic theorem.
\newblock {\em Duke Math. J.}, 133(1): 95--123, 2006.

\bibitem[DK78]{dekking_keane:1978}
F.M. Dekking and M.~Keane.
\newblock Mixing properties of substitutions.
\newblock {\em Z. Wahrscheinlichkeitstheorie und Verw. Gebiete}, 42(1): 23--33,
  1978.

\bibitem[DM08]{downarowich_maass:2008}
T.~Downarowicz and A.~Maass.
\newblock Finite rank {Bratteli-Vershik} diagrams are expansive.
\newblock {\em Ergod. Th. \& Dynam. Sys.}, 28(3): 739--747, 2008.

\bibitem[D10]{berthe:2010} F. Durand. Combinatorics on Bratteli diagrams and dynamical systems. In
\newblock{\em Combinatorics, Automata and Number Theory}. V. Berth\'e, M. Rigo
(Eds).
Encyclopedia of Mathematics and its Applications 135, Cambridge
University Press (2010), 338--386.


\bibitem[DHS99]{durand_host_scau:1999}
F.~Durand, B.~Host, and B.~Skau.
\newblock Substitutional dynamical systems, {Bratteli} diagrams and dimension
  groups.
\newblock {\em Ergod. Th. \& Dynam. Sys.}, 19: 953--993, 1999.

\bibitem[ES79]{effors_shen:1979} E.~Effros and C.~Shen. \newblock
Dimension groups and finite difference equations. \newblock {\em J.
Operator Theory}, 2:215--231, 1979.

\bibitem[ES81]{effros_shen:1981} E.~Effros and C.~Shen. \newblock The geometry of finite rank dimension groups.
 \newblock {\em Illinois J. Math.} 25(1):27–38, 1981.

\bibitem[F96]{ferenczi:1996}
S.~Ferenczi. \newblock Rank and symbolic complexity. \newblock{\em Ergodic Theory Dynam. Systems},
16(4):663--682, 1996.

\bibitem[F97]{ferenczi:1997} S.~Ferenczi. \newblock
Systems of finite rank. \newblock {\em
Colloq. Math.} 73(1):35--65, 1997.


\bibitem[FFT09]{ferenczi_fisher_talet}
S.~Ferenczi, A.M. Fisher, and M.~Talet.
\newblock Minimality and unique ergodicity of adic transformations.
\newblock {\em J. Anal. Math.}, 109(1): 1--31, 2009.

\bibitem[Fis09]{fisherUniqueErgodicity:preprint}
A.M. Fisher.
\newblock Nonstationary mixing and the unique ergodicity of adic
  transformations.
\newblock {\em Stoch. Dyn.} 9(3): 335--391, 2009.

\bibitem[GK07]{galatolo_kim:2007}
 S.~Galatolo, D.H.~Kim. \newblock
The dynamical Borel-Cantelli lemma and the waiting time problems. \newblock {\em Indag. Math. (N.S.)}
  18(3):421–434, 2007.

\bibitem[GJ02]{gjerde_johansen:2002}
R.~Gjerde and O.~Johansen.
\newblock Bratteli-Vershik models for Cantor minimal systems associated to
  interval exchange transformations.
\newblock {\em Math. Scand.}, 90(1): 87 -- 100, 2002.

\bibitem[GPS95]{giordano_putnam_skau:1995}
T.~Giordano, I.~Putnam, and C.~Skau.
\newblock Topological orbit equivalence and {$C^*$-crossed} products.
\newblock {\em J. Reine Angew. Math.}, 469: 51 -- 111, 1995.

\bibitem[GH82]{goodearl_handelman:82}
K.R.~Goodearl and D.~Handelman. 
\newblock Stenosis in dimension groups and AF $C^{\ast} $-algebras. 
\newblock{\em J. Reine Angew. Math.}, 332: 1–98, 1982.

\bibitem[Haj76]{hajnal:1976}
J.~Hajnal.
\newblock On products of non-negative matrices.
\newblock {\em Math. Proc. Cambridge Philos. Soc.}, 79(3): 521--530, 1976.

\bibitem [Han99]{handelman:1999} D.~Handelman.
\newblock Eigenvectors and ratio limit theorems for Markov chains and their relatives.
\newblock {\em J. Anal. Math.} 78: 61–116, 1999.

\bibitem[Har02]{hartfiel:book:2002}
D.~J. Hartfiel.
\newblock {\em Nonhomogeneous matrix products}.
\newblock World Scientific Publishing Co., 2002.

\bibitem[HPS92]{herman_putnam_skau:1992}
R.~H. Herman, I.~Putnam, and C.~Skau.
\newblock Ordered {Bratteli} diagrams, dimension groups, and topological
  dynamics.
\newblock {\em Int. J. Math.}, 3(6): 827 -- 864, 1992.

\bibitem[JB90]{johnson_bru:1990}
C.R. Johnson and R.~Bru.
\newblock The spectral radius of a product of nonnegative matrices.
\newblock {\em Linear Algebra Appl.}, 141: 227--240, 1990.

\bibitem[K80]{katok:1980}
A. Katok.
\newblock Interval exchange transformations and some special flows are not mixing.
\newblock {\em Israel J. Math.}, 35(4): 301--310, 1980.


\bibitem[Ke68]{keane:1968} M.~Keane, \newblock Generalized Morse sequences.
\newblock {\em Z. Wahrscheinlichkeitstheorie und Verw. Gebiete} 10: 335--353, 1968.

\bibitem[Ma77]{martin:1977} J.~Martin, \newblock The structure of generalized Morse minimal sets on $n$ symbols.
\newblock {\em Trans. Amer. Math. Soc.} 232: 343--355, 1977.

\bibitem[Me06]{medynets:2006}
K.~Medynets.
\newblock Cantor aperiodic systems and {Bratteli} diagrams.
\newblock {\em C. R., Math., Acad. Sci. Paris}, 342(1): 43--46, 2006.

\bibitem[Mel06]{mela:2006}
X. M\'ela,
\newblock A class of nonstationary adic transformations.
\newblock {\em Ann.\ I. H. Poincar\'e}, 42: 103--123, 2006.


\bibitem[Pul71]{pullman:1971}
N. J. Pullman.
\newblock A geometric approach to the theory of nonnegative matrices.
\newblock {\em Linear Algebra and Appl.}, 4: 297--312, 1971.

\bibitem[Ro84]{rosenthal:1984} A. Rosenthal. \newblock Les syst\`emes de rang fini exact ne sont pas m\'elangeants. \newblock{\em Preprint}, 1984.

\bibitem[Ru76]{rusev:1976} P.~Rusev, \newblock Hermite functions of second kind. \newblock {\em Serdica.} 2:177-190, 1976.

\bibitem[Sen81]{seneta:book:1981}
E.~Seneta.
\newblock {\em Non-negative matrices and Markov chains}.
\newblock Springer Series in Statistics. Springer-Verlag, New York, 1981.

\bibitem[VK81]{vershik_kerov:1981}
A.M. Vershik and S.V. Kerov.
\newblock Asymptotic theory of the characters of a symmetric group.
\newblock {\em Funktsional. Anal. i Prilozhen.}, 15(4): 15--27, 1981.
\newblock (Russian).

\bibitem[War02]{wargan:thesis:2002}
Krzysztof Wargan.
\newblock {\em $S$-adic Dynamical Systems and Bratteli diagrams}.
\newblock PhD thesis, George Washington University, 2002.

\end{thebibliography}
\end{document}